\documentclass[11pt]{amsart}
\usepackage{amsmath, amssymb, amsfonts}
\usepackage{graphicx, overpic}
\usepackage{fullpage}
\usepackage{color}

\theoremstyle{plain}
\newtheorem{thm}{Theorem}[section]
\newtheorem{lemma}[thm]{Lemma}
\newtheorem{prop}[thm]{Proposition}
\newtheorem{cor}[thm]{Corollary}

\numberwithin{equation}{section}
\numberwithin{figure}{section}

\theoremstyle{definition}
\newtheorem{definition}[thm]{Definition}
\newtheorem{example}[thm]{Example}
\newtheorem{remark}[thm]{Remark}

\newtheorem{examples}[thm]{Examples}
\newtheorem{obs}[thm]{Observation}
\newtheorem{claim}[thm]{Claim}

\renewcommand{\div}[1]{\mathrm{div}_{#1}}
\newcommand{\G}{\Gamma}
\newcommand{\Z}{\mathbb{Z}}
\newcommand{\R}{\mathbb{R}}

\newcommand{\avoid}[2]{d^{\mathrm{av}}(#1,#2)}
\newcommand{\Cayley}[1]{\mathcal {C}_#1}
\newcommand{\cfs}{$\mathcal{CFS}$}
\newcommand{\CAT}{\operatorname{CAT}}
\newcommand{\cS}{\mathcal{S}}
\newcommand{\la}{\langle}
\newcommand{\ra}{\rangle}
\newcommand{\s}{\sigma}
\newcommand{\diam}{\operatorname{diam}}

\newcommand{\hide}[1]{}

\title{Divergence in right-angled Coxeter groups}
\author{Pallavi Dani and Anne Thomas}
\thanks{The first author was supported in part by Louisiana Board of Regents 
Support Fund Contract LEQSF(2011-14)-RD-A-06
and NSF Grant No.~DMS-1207868.  This research of the second author was supported in part by ARC Grant No.~DP110100440, and the second author was also supported in part by an Australian Postdoctoral Fellowship.}

\begin{document}
\maketitle

\begin{abstract}  Let $W$ be a $2$-dimensional right-angled Coxeter group.  We characterise such $W$ with linear and quadratic divergence, and construct right-angled Coxeter groups with divergence polynomial of arbitrary degree.  Our proofs use the structure of walls in the Davis complex.
\end{abstract}

\section{Introduction}
The divergence of a pair of geodesics is a classical notion related to curvature.  
Roughly speaking, given a pair of geodesic rays emanating from a basepoint, their divergence 
measures, as a function of~$r$, the length of a shortest ``avoidant'' path connecting their time-$r$ points.  A path is \emph{avoidant} if it stays at least distance $r$ away from the basepoint. 
In~\cite{gersten-quadratic}, Gersten used this idea to define a quasi-isometry invariant of spaces, 
also called divergence.  We recall the definitions of both notions of divergence in Section~\ref{sec:divergence}.   

The divergence of every pair of geodesics in Euclidean space is a linear function, and it follows from Gersten's definition that any group quasi-isometric to Euclidean space has linear divergence.  
In a $\delta$-hyperbolic space, any pair of non-asymptotic rays diverges exponentially; thus the divergence of any hyperbolic group is exponential.  
In symmetric spaces of non-compact type, the divergence is either linear or exponential, and  
Gromov suggested in~\cite{gromov} the same should be true in 
 $\CAT(0)$ spaces.

Divergence has been investigated for many important groups and spaces, 
and contrary to Gromov's expectation, quadratic divergence is common.  
Gersten first exhibited quadratic divergence for certain $\CAT(0)$ spaces
in~\cite{gersten-quadratic}.  He 
then proved in~\cite{gersten-3mfld} 
that the divergence of the fundamental 
group of a closed geometric $3$-manifold is either linear, quadratic or exponential, and 
characterised the (geometric) ones with quadratic divergence as the fundamental groups of graph manifolds.  
Kapovich--Leeb~\cite{kapovich-leeb} showed that all graph manifold groups have quadratic divergence. More recently, Duchin--Rafi~\cite{duchin-rafi} established that 
the divergence of Teichm\"uller space and the mapping class group is quadratic (for mapping class groups this was also obtained by Behrstock in~\cite{behrstock}).
Dru\c{t}u--Mozes--Sapir~\cite{drutu-mozes-sapir} have conjectured that the divergence of lattices in higher 
rank semisimple Lie groups is always linear, and proved this conjecture in some cases.  
Abrams et al~\cite{abddy} and independently Behrstock--Charney~\cite{behrstock-charney} have shown that if $A_\G$ is the right-angled Artin group associated to a graph $\G$, 
the group $A_\G$ has either linear or quadratic divergence, and its divergence is linear if and only if 
$\G$ is (the $1$-skeleton of) a join.

In this work we study the divergence of $2$-dimensional right-angled Coxeter groups.  Our first main result is Theorem~\ref{thm:linear and quadratic} below, which characterises such groups with linear and quadratic divergence in terms of their defining graphs.  This result can be seen as a step in the quasi-isometry classification of (right-angled) Coxeter groups, about which very little is known.  

We note that by~\cite{davis-jan}, every right-angled Artin group is a finite index subgroup of, and therefore quasi-isometric to, a right-angled Coxeter group.
However, in contrast to the setting of right-angled Artin groups, where one sees only linear and quadratic divergence, even the class of $2$-dimensional right-angled Coxeter groups exhibits a greater variety of divergence functions.  For example, there exist 2-dimensional right-angled Coxeter groups that are hyperbolic, and therefore have exponential divergence.  
Our second main result provides further evidence of this phenomenon: in Theorem~\ref{thm:all degrees} below, we construct right-angled Coxeter groups with divergence polynomial of any degree.

Given a finite simplicial graph $\G$, the associated \emph{right-angled Coxeter group} $W_\G$ has generating set $S$ the vertices of $\G$, and relations $s^2 = 1$ for all $s \in S$ and $st = ts$ whenever $s$ and $t$ are adjacent vertices.   
We restrict attention to~$W_\G$ one-ended and of dimension $2$; equivalently, $\G$ is connected,  triangle-free and has no separating vertices or edges.  The group $W_\G$ acts geometrically on its Davis complex $\Sigma_\G$.  As $\Sigma_\G$ is a $\CAT(0)$ square complex, $W_\G$ is a $\CAT(0)$ group.  
We investigate divergence by considering geodesics and paths in the Cayley graph of $W_\G$ with respect to the generating set $S$.  This Cayley graph may be identified with the $1$-skeleton of the Davis complex $\Sigma_\G$, and we use many properties of walls in the Davis complex to determine upper and lower bounds on lengths of avoidant paths.   
See Section~\ref{sec:coxeter and davis} for details and further background on $W_\G$ and $\Sigma_\G$, including references.

By Moussong's Theorem~\cite[Corollary 12.6.3]{davis-book}, $W_\G$ is hyperbolic if and only if $\G$ has no embedded cycles of length four.   
In order to investigate divergence for $W_\G$ not hyperbolic, we consider the set of embedded four-cycles in $\G$.  Each such four-cycle induces a family of isometrically embedded flats in $\Sigma_\G$.  In Section~\ref{sec:linear and quadratic} we define an explicit, easy-to-check condition, which we call \cfs, on the graph $\G$.  If $\G$ is \cfs\ then $\Sigma_\G$ has a distinguished collection of flats coming from a specific class of four-cycles in $\G$, with these flats intersecting along infinite bands, such that each generator of $W_\G$ is in the four-cycle for at least one such flat.  

\begin{thm}\label{thm:linear and quadratic}  Let $\G$ be a finite, simplicial, connected, triangle-free graph which has no separating vertices or edges.  Let $W_\G$ be the associated right-angled Coxeter group.  
\begin{enumerate}
\item\label{part:linear} The group $W_\G$ has linear divergence if and only if $\G$ is a join.
\item\label{part:quadratic} The group $W_\G$ has quadratic divergence if and only if $\G$ is \cfs\ and is not a join.
\end{enumerate}
\end{thm}

Note that part~\eqref{part:linear} is equivalent to saying that $W_\G$ has linear divergence if and only if it is reducible, since for $\G$ triangle-free, $W_\G$ is reducible if and only if $\G$ is a join.  Our proof of part~\eqref{part:linear} 
is similar to that of the corresponding result for $A_\G$ in~\cite{abddy}.

To establish a quadratic upper bound on divergence when the graph $\G$ is \cfs, we construct,
given a pair of geodesic segments 
based at a common point, an avoidant path between their endpoints which travels only in flats from the distinguished collection of flats guaranteed by the \cfs\ condition.  Since the divergence within a flat is linear, the quadratic upper bound comes from showing that this path only needs to pass through linearly many flats.  
As pointed out by the referee,
this quadratic upper bound could also be obtained using the thickness machinery developed by Behrstock--Dru\c{t}u~\cite{behrstock-drutu}. (See Remark~\ref{rmk:thick-cfs}.)

The more delicate direction of part~\eqref{part:quadratic} of Theorem~\ref{thm:linear and quadratic} is proving that \cfs\ graphs are 
exactly the class of graphs for which there is a quadratic upper bound on divergence.  We in fact establish a cubic lower bound on divergence when $\G$ is not \cfs.   
To obtain lower bounds on the lengths of avoidant paths, we consider van Kampen diagrams 
whose boundaries consist of a pair of geodesic segments with common basepoint and an avoidant path between their endpoints.  
The fact that the defining graph is not \cfs\ has certain implications on the cell-structure of the van Kampen diagram, which force a lower bound on the length of its boundary (and therefore of the avoidant path).

In contrast with the classes of groups discussed above, right-angled Coxeter groups may have divergence other than linear, quadratic or exponential.  We prove:

\begin{thm}\label{thm:all degrees}  For all $d \geq 1$, there is a right-angled Coxeter group $W_d$ with divergence polynomial of degree $d$.
\end{thm}

\noindent In~\cite{gersten-3mfld}, Gersten asked 
whether polynomial divergence of degree $\geq 3$ is possible for $\CAT(0)$ groups.  
Macura~\cite{macura} constructed a family of $\CAT(0)$ groups $G_d$ with divergence polynomial of degree $d \geq 2$.  These groups $G_d$ are the same as the ``hydra groups" investigated by Dison--Riley~\cite{dison-riley}. 
Behrstock--Dru\c{t}u~\cite{behrstock-drutu} subsequently obtained examples of $\CAT(0)$ groups $H_d$ with divergence polynomial of any degree $d \geq 2$, with $H_d$ the amalgamated free product of two copies of $H_{d-1}$ along an infinite cyclic subgroup.  The groups $W_d$ that we construct are not of this form.  Most recently, Behrstock--Hagen~\cite{behrstock-hagen} used a similar construction to that of~\cite{behrstock-drutu} to obtain fundamental groups of $\CAT(0)$ cube complexes with divergence polynomial of any degree.
Theorem~\ref{thm:all degrees} provides an answer to Gersten's question within a well-known class of 
$\CAT(0)$ groups.

We prove Theorem~\ref{thm:all degrees} in Section~\ref{sec:higher-degree}, where we inductively construct a family of graphs $\G_d$ such that $W_d = W_{\G_d}$ has divergence polynomial of degree $d$.  
We prove upper and lower bounds on the divergence of $W_d$ 
in Propositions~\ref{prop:poly-upper} and~\ref{prop:poly-lower} respectively.
As discussed in Remark~\ref{rmk:thick}, the upper bound for the divergence of $W_d$ could also be derived from thickness considerations. 
Our arguments to obtain the lower bounds on divergence are considerably shorter than Macura's.  

After proving Theorem~\ref{thm:all degrees}, we noticed that Macura's group $G_d$ and our group $W_d$ both act geometrically on a $\CAT(0)$ square complex with all vertex links 
equal to the graph~$\G_d$ (namely the Cayley $2$-complex for $G_d$, and the Davis complex for $W_d
$, respectively).  A natural question is thus whether $G_d$ and $W_d$ are commensurable.  
Since our techniques for addressing this question are quite different to those used to prove 
Theorems~\ref{thm:linear and quadratic} and~\ref{thm:all degrees}, we discuss this question in 
Appendix~\ref{app:macura}.  We first show in Proposition~\ref{prop:commens}, using covering theory and complexes of groups, that 
$G_2$ and $W_2$ are commensurable.   While attempting to prove commensurability of $G_d$ 
and $W_d$ for $d > 2$, we were surprised to discover that their corresponding  square complexes 
are not in fact isometric (see Corollary~\ref{cor:not isometric}).  
Hence the strategy of finding a common finite cover to establish commensurability fails.  We do not 
know whether $G_d$ and $W_d$ are commensurable or even quasi-isometric for~$d > 2$.

\subsection*{Acknowledgements}  We thank the University of Sydney for travel support.  We also thank the organisers of the 2012 Park City Mathematics Institute Summer Program on Geometric Group Theory, during which part of this work was undertaken, and an anonymous referee for helpful comments.

\section{Divergence}\label{sec:divergence}
In this section we recall Gersten's definition of divergence as a quasi-isometry invariant 
from~\cite{gersten-quadratic}.  We restrict to spaces which are one-ended.  

Let $(X, d)$ be a one-ended geodesic metric space.  For $p \in X$, let $S(p,r)$ and $B(p,r)$ denote the sphere and open ball of radius $r$ about $p$.  A path in $X$ is said to be \emph{$(p,r)$-avoidant} if it lies in 
$X-B(p,r)$.  Then, given a pair of points $x,y \in X-B(p,r)$, the \emph{$(p,r)$-avoidant distance} $d^{\mathrm{av}}_{p,r}(x,y)$
between them is the infimum of the lengths of all $(p,r)$-avoidant paths connecting $x$ and $y$.  

Now fix a basepoint $e\in X$.  In the rest of the paper we will write \emph{$r$-avoidant} or simply \emph{avoidant} for 
$(e,r)$-avoidant, and  
$\avoid{x}{y}$ for $d^{\mathrm{av}}_{e,r}(x,y)$, indicating the basepoint and radius only if they differ from 
$e$ and $r$.  

For each $0 <\rho \le 1$, let 
$$
\delta_{\rho}(r) = \sup_{x, y \in S(e,r)} d^{\mathrm{av}}_{\rho r}(x,y).
$$
Then the \emph{divergence} of $X$ is defined to be the resulting collection of functions 
$$
\div{X} = \{\delta_{\rho}
\,|\, 0 < \rho
\le 1 \}.
$$  

The spaces $X$ that we will consider (Cayley graphs of right-angled Coxeter groups)  have the geodesic extension property (i.e.~any finite geodesic segment can be extended to an infinite geodesic ray).  It is not hard to show that in a metric space $X$ with this property, 
$\delta_{\rho} \simeq \delta_{1}$ for all $0 < \rho \le 1$, 
 where $\simeq$ is the equivalence on functions generated by:
 $$
 f \preceq g \iff \; \exists \; C>0 \text{ such that } f(r) \le Cg(Cr+C) + Cr+C.  
 $$
Thus in this paper, we think of $\div{X}$ as a function of $r$, defining it to be equal to $\delta_1$.   We say that the divergence of $X$ is \emph{linear} if $\div{X}(r)\simeq r$, \emph{quadratic} if $\div{X}(r)\simeq  r^2$, and so on.  

The divergence of $X$ is then, up to the relation $\simeq$, a quasi-isometry invariant which is independent of the chosen basepoint (see~\cite{gersten-quadratic}).    Thus it makes sense to define the divergence of a finitely generated group to be the divergence of one of its Cayley graphs.  

The divergence of 
a pair of geodesic rays $\alpha$ and $\beta$ with the same initial point $p$, or of a bi-infinite geodesic $\gamma$ 
are defined as, respectively,  
$$
\div{\alpha, \beta}(r) = d^{\mathrm{av}}_{p,r}(\alpha(r), \beta(r)) \text{ and } 
\div{\gamma}(r) = d^{\mathrm{av}}_{\gamma(0),r}(\gamma(-r), \gamma(r)).
$$

Note that in a geodesic metric space $X$, if $\div{\alpha, \beta}(r) \leq f(r)$ for all pairs of geodesic rays in $X$ with initial point $e$, then 
$\div{X}(r) \leq f(r)$. On the other hand, if there exists a pair of geodesic rays (or a bi-infinite geodesic) such 
that $\div{\alpha, \beta}(r) \succeq f(r)$, then $\div{X}(r) \succeq f(r)$.  
Finally, if $X$ is CAT(0) and $\div{\alpha, \beta}(r) \ge f(r)$,
then, using the fact that projections do not increase
distances, one can show that 
$d^{\mathrm{av}}_{p,r}(\alpha(s), \beta(t)) \ge f(r)$ for any $s, t \geq r$.  
These observations will be used repeatedly in proofs.  

\section{Coxeter groups and the Davis complex}\label{sec:coxeter and davis}
In this section, we recall definitions and results concerning right-angled Coxeter groups 
(Section~\ref{sec:coxeter}) and their associated Davis complexes (Section~\ref{sec:davis}).  
Section~\ref{sec:walls} then gives a careful discussion of walls in the Davis complex. 
Section~\ref{sec:paths} 
discusses paths in the Cayley graph of $W_\G$ and their 
relationship to walls in the Davis complex.  We mostly follow Davis' book~\cite{davis-book}.

\subsection{Right-angled Coxeter groups}\label{sec:coxeter}

Let $\G$ be a finite simplicial graph with vertex set $S$ and let $W_\G$ be the associated right-angled Coxeter group, as defined in the introduction.  The group $W_\G$ is \emph{reducible} if $S$ can be written as a disjoint union $S_1 \sqcup S_2$ of nonempty subsets such that $W_1:=\langle S_1 \rangle$ commutes with $W_2 := \langle S_2 \rangle$, in which case $W = W_1 \times W_2$.   

In this paper we restrict to $\G$ triangle-free.  Then it is easy to see that $W_\G$ is reducible if and only if $\G$ is a join (i.e.~a complete bipartite graph).   Also, with this assumption,  $W_\G$ is one-ended if and only if $\G$ is connected and has no separating vertices or edges (see Theorem 8.7.2 of~\cite{davis-book}).

Given $T \subseteq S$, the subgroup $W_T:= \langle T \rangle$ of $W_\G$ is called a \emph{special subgroup}.    By convention, $W_\emptyset$ is the trivial group.   If $\Lambda$ is an induced subgraph of $\G$ with vertex set $T$, we may write $W_\Lambda$ for the special subgroup $W_T$.  Denote by~$C_2$ the cyclic group of order $2$ and by $D_\infty$ the infinite dihedral group.  Then for each $s \in S$, the special subgroup $W_{\{s\}}$ is isomorphic to $C_2$.   If $s$ and $t$ are adjacent vertices, then $W_{\{s,t\}} \cong C_2 \times C_2$, while if $s$ and $t$ are non-adjacent vertices, we have $W_{\{s,t\}} \cong D_\infty$.

\begin{example}\label{ex:fourcycle}
Suppose
$T = \{ s,t,u,v\} \subset S$ is such that $s$, $t$, $u$ and $v$ are, in cyclic order, the vertices of an embedded four-cycle in $\G$.  Then $W_T$ is reducible with $$W_T = W_{\{s,u\}} \times W_{\{t,v\}} \cong D_\infty \times D_\infty.$$  Now suppose $T_1$ and $T_2$ are distinct subsets of $S$ such that $T_1 \cap T_2 = \{ s,t,u\}$, with $s$ and $u$ both adjacent to $t$. 
Since $\G$ is triangle-free, this implies that $s$ and $u$ are not connected by an edge.  Then $$ W_{T_1 \cap T_2} = W_{\{ s, u \}} \times W_{\{ t\}} \cong D_\infty \times C_2$$ and 
$W_{T_1 \cup T_2}$ splits as the amalgamated free product $$W_{T_1 \cup T_2} = W_{T_1} *_{W_{T_1 \cap T_2}} W_{T_2} \cong W_{T_1} *_{D_\infty \times C_2} W_{T_2}.$$ 
\end{example}

A special subgroup $W_T$ is said to be a \emph{spherical} special subgroup if $W_T$ is finite.  The set of \emph{spherical subsets} of $S$, denoted $\cS$, is the set of subsets $T \subseteq S$ such that $W_T$ is spherical.  (The reason for the terminology ``spherical" is that if $W_T$ is finite, then $W_T$ acts as a geometric reflection group on the unit sphere in $\mathbb{R}^{|T|}$; see Theorem 6.12.9 of \cite{davis-book}.)   It follows from the paragraph before Example~\ref{ex:fourcycle} that for $\G$ triangle-free, the only spherical subsets of~$S$ are the empty set, the sets $\{s \}$ for $s \in S$, and the sets $\{s,t\}$ where $s$ and $t$ are adjacent vertices.  The corresponding spherical special subgroups of $W$ are isomorphic to the trivial group, $C_2$, and $C_2 \times C_2$ respectively.

A \emph{word} in the generating set $S$ is a finite sequence $\mathbf{s} =
(s_1,\ldots,s_k)$ where each $s_i \in S$.  We denote by $w(\mathbf{s})= s_1
\cdots s_k$ the corresponding element of $W$.  The \emph{support} of a word $\mathbf{s}$ is the set of generators which appear in $\mathbf{s}$.  A word $\mathbf{s}$ is said to be \emph{reduced} if the element
$w(\mathbf{s})$ cannot be represented by any shorter word, and a word $\mathbf{s}$ is \emph{trivial} if $w(\mathbf{s})$ is the trivial element.  We will later by abuse of notation write $s_1 \cdots s_k$ for both words and group elements.  A word $\mathbf{s}$ in the generating set $S$ of a right-angled Coxeter group is reduced if and only if it cannot be shortened by a sequence of operations of either deleting a subword of the form $(s,s)$, with $s \in S$, or replacing a subword $(s,t)$ such that $st = ts$ by the subword $(t,s)$. (This is a special case of Tits' solution to the word problem for Coxeter groups; see Theorem 3.4.2 of~\cite{davis-book}.)  

\subsection{The Davis complex}\label{sec:davis}

From now on, $\G$ is a finite, simplicial, connected, triangle-free graph with no separating vertices or edges, and $W=W_\G$ is the associated right-angled Coxeter group.  In this section, we discuss the Davis complex for $W$.  

By our assumptions on $\G$, we may define the \emph{Davis complex} $\Sigma = \Sigma_\G$ to be the Cayley $2$-complex for the presentation of $W_\G$ given in the introduction, in which all disks bounded by a loop with label $s^2$ for $s \in S$ have been shrunk to an unoriented edge with label $s$.  Then the vertex set of $\Sigma$ is $W_\G$ and the $1$-skeleton of $\Sigma$ is the Cayley graph $\Cayley\G$ of $W$ with respect to the generating set $S$.  Since all relators in this presentation other than $s^2 = 1$ are of the form $stst = 1$, $\Sigma$ is a square complex.  We call this cellulation of $\Sigma$ the \emph{cellulation by big squares}, with the \emph{big squares} being the $2$-cells.  Note that the link of each vertex in this cellulation is the graph $\G$.  

We next define the \emph{cellulation by small squares} of $\Sigma$ to be the first square subdivision of the cellulation by big squares, with the \emph{small squares} being the squares obtained on subdividing each big square into four.  We will use both of these cellulations in our proofs.

We now assign \emph{types} $T \in \cS$ to the vertices of the cellulation by small squares.   If $\s$ is also a vertex of the cellulation by big squares, then $\s$ has type $\emptyset$.  If $\s$ is the midpoint of an edge in the cellulation by big squares, then since $\Cayley\G$ is the $1$-skeleton of the cellulation by big squares, $\s$ is the midpoint of an edge connecting $g$ and $gs$ for some $g \in W$ and $s \in S$, and we assign type $\{ s\} \in \cS$ to $\s$.  Finally if $\s$ is the centre of a big square, then $\s$ is assigned type $\{s, t \} \in \cS$, where two of the vertices adjacent to $\s$ have type $\{ s\}$, and two of the vertices adjacent to $\s$ have type $\{t\}$.  

Consider $\Sigma$ with the cellulation by small squares.  The group $W$ naturally acts on the left on $\Sigma$, preserving types, so that the stabiliser of each vertex of type $T \in \cS$ is a conjugate of the finite group $W_T$.  Let $\s$ be the vertex of type $\emptyset$ corresponding to the identity element of $W$.  The \emph{base chamber} $K$ is the union of the set of small squares which contain $\s$.  Any translate of $K$ by an element of $W$ is called a \emph{chamber}.  For each $T \in \cS$, we denote by $\s_T$ the unique vertex of type $T \in \cS$ in the base chamber.  The quotient of $\Sigma$ by the action of $W$ is the base chamber $K$, and the $W$-stabiliser of $\s_T$ is precisely the spherical special subgroup $W_T$. 

For $s \in S$, the \emph{mirror} $K_s$ is the union of the set of edges in the base chamber which contain $\s_{\{s\}}$ but not $\s_\emptyset$.  The mirror $K_s$ is thus the star graph of valence $n$, where $n$ is the cardinality of the set  $\{ t \in S \mid st = ts, t \neq s \} $.  Note that $n \geq 2$, since $\G$ has no isolated vertices or vertices of valence one.  The \emph{centre} of the mirror $K_s$ is the vertex $\s_{\{s\}}$.  Any translate of $K_s$ by an element of $W$ is called a \emph{panel (of type $s$)}.  

Let $\Sigma$ be the Davis complex cellulated by either big or small squares.  We now metrise $\Sigma$ so that each big square is a unit Euclidean square, hence each small square is a Euclidean square of side length $\frac{1}{2}$.  By~\cite[Theorem 12.2.1]{davis-book}, this piecewise Euclidean cubical structure on $\Sigma$ is $\CAT(0)$.  Since the group $W$ acts on $\Sigma$ with compact quotient $K$ and finite stabilisers, $W$ is a $\CAT(0)$ group.    

Let $W_T$ be a special subgroup of $W$.   Then the Cayley graph of $W_T$ (with respect to the generating set $T$) embeds isometrically in $\Cayley{\G} \subset \Sigma$.  Hence for each $g \in W$ and each special subgroup $W_T$ of $W$, left-multiplication of the Cayley graph of $W_T$ by $g$ results in an isometrically embedded copy of the Cayley graph of $W_T$ in $\Cayley{\G} \subset \Sigma$, which contains the vertex $g$.  We will refer to this copy as the \emph{Cayley graph of $W_T$ based at $g$}.  For each special subgroup $W_T$ of $W$, and each coset $gW_T$, there is also an isometrically embedded copy of $\Sigma_T$ in $\Sigma$.  
If $\Theta$ is an induced subgraph of $\G$, 
we may denote by $\Sigma_\Theta$ the Davis complex for the special subgroup $W_\Theta$, and by
$\Cayley\Theta$ the Cayley graph for $W_\Theta$ with generating set the vertices of $\Theta$.  

\begin{remark}\label{rmk:davis}
Suppose that $T$ is the set of vertices of an embedded four-cycle in $\G$, so that $W_T \cong D_\infty \times D_\infty$.  Then each copy of $\Sigma_T$ in $\Sigma$ is an isometrically embedded copy of the Euclidean plane (tessellated by either big or small squares).  Consider $\Sigma$ with the cellulation by big squares and let $T_1$ and $T_2$ be the sets of vertices of embedded four-cycles in $\G$ such that $W_{T_1 \cup T_2}$ splits over $W_{T_1 \cap T_2} \cong D_\infty \times C_2$.  Then each intersection of a copy of the flat $\Sigma_{T_1}$ with a copy of the flat  $\Sigma_{T_2}$ in $\Sigma$ is an infinite band of big squares corresponding to a copy of 
$\Sigma_{T_1 \cap T_2}$.  To be precise, this infinite band of big squares is the direct product $\R \times [0,1]$ tessellated by squares of side length $1$.
\end{remark}

\subsection{Walls}\label{sec:walls}
Consider the Davis complex $\Sigma = \Sigma_\G$ with the cellulation by small squares.  Recall that an element $r \in W = W_\G$ is a \emph{reflection} if $r = gsg^{-1}$ for some $g \in W$ and $s \in S$.  A \emph{wall} in $\Sigma$ is defined to be the fixed set of a reflection $r \in W$.  For each reflection $r$, the wall associated to $r$ separates $\Sigma$, and $r$ interchanges the two components of the complement.  Each wall is a totally geodesic subcomplex of the $\CAT(0)$ space $\Sigma$, hence each wall is contractible.  By the construction of $\Sigma$, each wall in $\Sigma$ is a union of panels, and so is contained in the $1$-skeleton of $\Sigma$.  Hence each wall of $\Sigma$ is a tree.

We now assign \emph{types} $s \in S$ to the walls.  To show that this can be done in a well-defined fashion, suppose first that $gsg^{-1} = s'$, where $g \in W$ and $s,s' \in S$.  Fix a reduced word $(s_1,\ldots,s_k)$ for $g$, and consider the trivial word $\mathbf{s}= (s_1,\ldots,s_k,s, s_k, \ldots, s_1, s')$, which corresponds to the equation $gsg^{-1}s' = 1$.  Since $\mathbf{s}$ is non-reduced, by Tits' solution to the word problem for $W$ (see the final paragraph of Section~\ref{sec:coxeter} above), we must be able to reduce $\mathbf{s}$ to the empty word by a sequence of operations of deleting repeated letters, and swapping $ut$ for $tu$, where $u,t \in S$ are adjacent vertices.  It follows that the number of instances of each letter in $\mathbf{s}$ must be even.   Thus~$s = s'$, in other words, no two distinct elements of $S$ are conjugate in $W$.  Hence for any reflection $r \in W$, there is a unique $s \in S$ so that $r = gsg^{-1}$ for some $g \in W$.  It is thus well-defined to declare the type of the wall which is the fixed set of the reflection $r = gsg^{-1}$ to be $s$.  A wall of type $s$ is a union of panels of type $s$, and in fact is a maximal connected union of panels of type $s$.  So if each panel of type $s$ is a star-graph of valence $n \geq 2$, each wall of type $s$ will be a $(2,n)$-biregular tree.  

For each generator $s \in S$, we denote by $H_s$ the unique wall of type $s$ which contains a panel of the base chamber, and by $gH_s$, for $g \in W$, the unique translate of the wall $H_s$ which contains a panel of the chamber $gK$.  If $H$ is a wall of type $s$, then all walls that intersect $H$ are of types which commute with $s$ (and are not equal to $s$).  Since $\G$ is triangle-free, there are no triples of pairwise intersecting walls.  All intersections of walls consist of two walls intersecting at right angles at the centre of some big square, thus subdividing it into four small squares.  

\subsection{Paths}\label{sec:paths}
A \emph{path} in $\Cayley \G$ is a  
map from an interval (finite or infinite) to $ \Cayley\G$, such that each integer is mapped to a vertex 
of $\Cayley \G$ and consecutive integers are mapped to adjacent vertices.  
Given a path $\alpha$, we may use $\alpha(i)$ to denote either the image vertex in $\Cayley \G$ or the group element in $W_\G$ associated with that vertex.   

As noted in Section~\ref{sec:davis}, the Cayley graph $\Cayley{\G}$ is the $1$-skeleton of the cellulation of $\Sigma$ by big squares. 
In this cellulation, each edge of $\Cayley\G$ crosses a unique wall in $\Sigma$.   Thus  the length of a path in $\Cayley\G$ is equal to its number of wall-crossings (note that a path may cross a given wall more than once).  We will sometimes describe paths using the labels of the walls they cross.  For example, by the statement ``$\alpha$ is the geodesic ray emanating from (or based at) $g$ labelled $a_1 a_2 a_3 \dots$'' we will mean that $\alpha$ is a geodesic path such that $\alpha(0)=g$ and $\alpha(i)=ga_1a_2 \dots a_i$ for $i>0$.  The path will be a geodesic if each subsegment 
$a_i \dots a_j$ is reduced.  
We will often use the fact that a path is a geodesic if and only if it does not cross any wall twice (compare Lemma 3.2.14 and Theorem 3.2.16 of~\cite{davis-book}).  
If $\alpha$ is a geodesic, we will use the notation $\alpha_{[i_1,i_2]}$ to denote the part of $\alpha$ that lies between 
$\alpha(i_1)$ and $\alpha(i_2)$, including these endpoints.  The \emph{support} of a path is the set of labels of the walls that it crosses. 

Since $\G$ is  triangle-free, the set of all generators that commute with a given one, say $a$, generate a special subgroup $W_T$ of $W_{\G}$ which is a free product of finitely many copies of $C_2$.  Thus the Cayley graph of $W_T$ (with generating set $T$) is a tree.  Now consider a wall $gH_a$ of type $a$.  There is a copy of the Cayley graph of $W_T$ based at $g$ which runs parallel to the wall $gH_a$, at constant distance $\frac{1}{2}$ from this wall.  We say that a path emanating from $g$ \emph{runs along} the wall $gH_a$ if it is a path in this copy of the Cayley graph of $W_T$.  Equivalently,  the path emanates from $g$ and has support contained in the set of generators labelling the link of $a$ in $\G$.   

Another fact that will be used repeatedly is the following: Suppose $\gamma$ is a geodesic segment, and $\eta$ is any path between its endpoints.  Let $H$ be a wall that is crossed by $\gamma$.  Then $\eta$ crosses $H$ at least once.  This is because $H$ (like any wall) separates the Davis complex, and $\gamma$, being a geodesic, crosses $H$ exactly once.  Thus the endpoints of $\gamma$ are in different components of the complement of $H$.  Since $\eta$ is a (continuous) path connecting them, $\eta$ must cross $H$.

\section{Linear and quadratic divergence in right-angled Coxeter groups}\label{sec:linear and quadratic}

In this section we prove Theorem~\ref{thm:linear and quadratic} of the introduction.  We characterise the defining graphs of $2$-dimensional right-angled Coxeter groups with linear and quadratic divergence in Sections~\ref{sec:linear} and~\ref{sec:quadratic} respectively.  

All the graphs $\G$ considered in this section satisfy our standing assumptions: they are connected, simplicial, triangle-free and have no separating vertices or edges.  
Recall from Section~\ref{sec:divergence} that the divergence of $W_\G$ is by definition the divergence of one of its Cayley graphs.  
We denote by $\div{\G}$ the divergence of the Cayley graph $\Cayley\G \subset \Sigma_\G$.  All distances below will be measured in the Cayley graph $\Cayley\G$, that is, using the word metric on $W_\G$ with respect to the generating set $S$, and all paths considered will be in $\Cayley\G$.
\subsection{Linear divergence}\label{sec:linear}  

In this section we prove the following result.

\begin{thm}\label{thm:linear-div}
The divergence $\div{\G}$ is linear (i.e.~$\div{\G}(r)\simeq r $) if and only if $\G$ is a join. 
\end{thm}

As noted in Section~\ref{sec:coxeter}, the graph $\G$ is a join if and only if $W_\G$ is reducible (that is, $W$ splits as a direct product of special subgroups).  
It is proved in~\cite[Lemma 7.2]{abddy} that a direct product $H \times K$ has linear divergence if both $H$ and $K
$ have the geodesic extension property.
This property certainly holds for right-angled Coxeter groups.  
Thus if $\G$ is a join, $W_\G$ has linear divergence.

In Proposition~\ref{prop:quadratic-lower} below, we prove that when $\G$ is not a join, the Cayley graph of 
$W_{\G}$ contains a bi-infinite geodesic $\gamma$ such that $\div{\gamma}(r) \succeq r^2$. 
This completes the proof of Theorem~\ref{thm:linear-div}, as $\div{\G}(r) \succeq r^2$ in this case.  

\begin{definition}[The word $w$ and bi-infinite geodesic $\gamma$]\label{def:gamma}
Recall that the \emph{complementary graph} of~$\G$, 
denoted by $\G^c$, is the graph with the same vertex set as $\G$, in which two vertices are 
connected by an edge if and only if they are not connected by an edge in $\G$.  Since $\G$ is not a 
join, $\G^c$ is connected.  Choose a loop in $\G^c$ which visits each vertex (possibly with 
repetitions). Choose a vertex $a_1$ on this loop, and let $w=a_1 \dots a_k$ be 
the word formed by the vertices of this loop in the order encountered along the loop, where $a_k$ is 
the last vertex encountered before the loop closes up at its starting point $a_1$.  
We assume that the loop is never stationary at a vertex, so that $a_i \ne a_{i+1}$ for any $i$. 
Then $w$ is a 
word in the generators of $\G$ such that no two consecutive generators commute, and $a_k$ does 
not commute with $a_1$.  It follows that $w^n$ is reduced for all $n \in \Z$.  Let $\gamma$ be the 
bi-infinite geodesic in $\Cayley{\G}$ which passes through $e$ and is labelled by 
$\dots wwww\dots$, so that $\gamma(0)=e$, $\gamma(i)=a_1\cdots a_i$ for $1 \leq i \leq k$, $\gamma(-1)=a_k$, and so on.  
\end{definition}

\begin{prop}\label{prop:quadratic-lower}
If $\G$ is not a join, and 
$\gamma$ is the bi-infinite geodesic in $\Cayley{\G}$ from Definition~\ref{def:gamma}, then $\div{\gamma} (r)
\succeq r^2$.  
\end{prop}

The idea of the proof is similar to that of the corresponding result for right-angled Artin groups  in 
Lemma~7.3 in~\cite{abddy}, although we write it in terms of crossings of walls rather than 
van Kampen diagrams.  We include the proof here because it sets the stage for the proof of 
Proposition~\ref{prop:non-cfs-cubic-lower}. 

\begin{proof}
It is enough to obtain a lower bound on $\avoid{\gamma(-nk)} {\gamma(nk)}$ as a quadratic  function of $n$ (where $k$ is the length of the word $w$ from Definition~\ref{def:gamma}).  
  Let $\eta$ be an arbitrary avoidant path from $\gamma(-nk)$ to $\gamma(nk)$.  
Since $\gamma_{[-nk,nk]}$ is a geodesic and $\eta$ is a path with the same endpoints, $\eta$
must cross each wall crossed by $\gamma$ at least once.  
For notational convenience, we will focus on the walls $w^iH_{a_1}$ for $0 \le i \le n-1$ which are crossed by 
$\gamma_{[0,nk]}$.
Now let $(g_i, g_ia_1)$ be the edge of $\Cayley{\G}$ at which $\eta$ first crosses $w^i H_{a_1}$, where $g_i$ is the vertex in the component of the complement of $w^iH_{a_1}$ containing $e$.
Let $\eta_i$ be the part of $\eta$ between $g_{i}$ and $g_{i+1}$ (so that the first edge of $\eta_i$ is 
$(g_i, g_ia_1)$).  

For $0 \le i \le n-1$, 
let $\nu_i$ denote the geodesic connecting 
$w^i$ and $g_i$ which runs along $w^iH_{a_i}$, and let $H_i$ be the first wall crossed by $\nu_i$, with type $a_j$ for 
some $j$.  We claim that $H_i$ does not intersect $\nu_{i+1}$.  Since $a_j$ belongs to the support of $w$, 
the segment of $\gamma$ between $w^i$ and $w^{i+1}$ crosses a wall of type $a_j$.  By the construction of $w$, 
this wall cannot intersect $w^iH_{a_1}$.  It is therefore distinct from $H_i$ and consequently separates $H_i$ from 
$\nu_{i+1}$.  It follows that no subsequent wall crossed by $\nu_i$ intersects $\nu_{i+1}$ either.  Thus each wall crossed by $\nu_i$ separates $g_i$ and $g_{i+1}$ into distinct components.  Since $\eta_i$ is a path from 
$g_i$ to $g_{i+1}$, it must cross all of these walls.  Thus $\ell(\eta_i) \ge \ell(\nu_i) \ge k(n-i)$, and 
$$\ell(\eta) \ge \sum_{i=0}^{n-1} \ell(\eta_i) \ge \sum_{i=0}^{n-1} k(n-i) \ge \frac k2 n^2$$
which completes the proof.
\end{proof}

\subsection{Quadratic divergence}\label{sec:quadratic}

We first introduce the \cfs\ terminology for the graphs which give rise to right-angled 
Coxeter groups with quadratic divergence.  
The main result of this section is Theorem~\ref{thm:cfs-div}. below.

Given a graph $\G$, define the associated \emph{four-cycle graph} $\G^4$ as follows.  The vertices of $\G^4$ are 
the embedded loops of length four (i.e.~\emph{four-cycles}) in $\G$.  Two vertices of $\G^4$ are connected by an 
edge if the corresponding four-cycles in $\G$ share a pair of adjacent edges.   For example, if $\G$ is the 
join $K_{2,3}$,
then $\G^4$ is a triangle.
Given a subgraph $\Theta$ of $\G^4$, we define the \emph{support} of $\Theta$ 
to be the collection of vertices of $\G$ (i.e. generators of $W_{\G}$) that appear in the four-cycles in $\G$ corresponding to the vertices of $\Theta$. 

\begin{definition}[\cfs]
A graph $\G$ is said to be \cfs\   if there exists a component of $\G^4$ whose support is the entire vertex set of $\G$, i.e., there is a ``$\mathcal C$omponent with $\mathcal F$ull $\mathcal S$upport''.
\end{definition}

\begin{figure}
\begin{overpic}[scale=1]%
{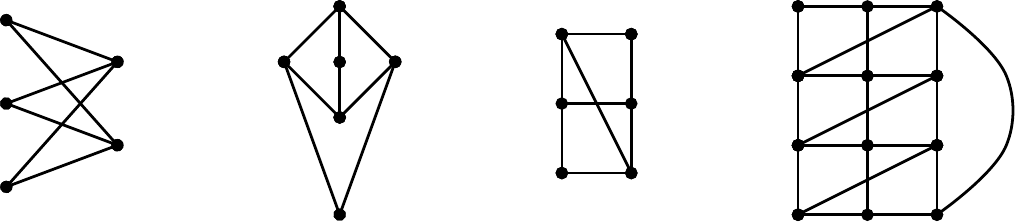}
\end{overpic}
\caption{Some \cfs\  graphs. (The middle two are actually the same graph.)}
\label{fig:cfs}
\end{figure}

\begin{figure}
\begin{overpic}[scale=1]%
{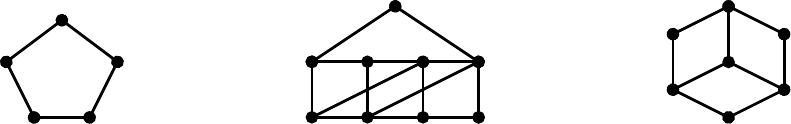}
\end{overpic}
\caption{Some non-\cfs\ graphs. (The four-cycle graph of the first one is empty.  In the second one the four-cycle graph is connected but does not have full support, 
while in the third, the four-cycle graph has full support, but is not connected and does not have a component with full support.)}
\label{fig:noncfs}
\end{figure}

Figures~\ref{fig:cfs} and~\ref{fig:noncfs} show some examples of \cfs\ graphs and non-\cfs\ graphs 
respectively.  
Note that any join is \cfs.  
The last example in Figure~\ref{fig:cfs} shows that the four-cycle graph of a \cfs\ 
graph need not be connected. However, the following observation will be useful in what follows:

\begin{obs} \label{obs:cfs}
The graph $\G$ is \cfs\ if and only it if has a subgraph $\Lambda$ such that $\Lambda^4$ is connected, and the support of 
$\Lambda^4$ is the vertex set of $\G$.  The graph $\Lambda$ is obtained from $\G$ by (possibly) deleting some edges, while keeping all the vertices.  
\end{obs}

We now characterise the graphs which give rise to right-angled Coxeter groups with quadratic divergence.

\begin{thm}\label{thm:cfs-div}
The divergence $\div \G$ is quadratic (i.e.~$\div{\G}(r)\simeq r^2 $) if and only if $\G$ is \cfs\ and not a join. 
\end{thm}

In Proposition~\ref{prop:cfs-quadratic-upper} below 
we obtain a quadratic upper bound on $\div \G$ when $\G$ is a \cfs\ graph.
On the other hand, Proposition~\ref{prop:quadratic-lower} above shows that if $\G$ is not a join, then there is a quadratic lower bound on $\div \G$.  This proves one direction of Theorem~\ref{thm:cfs-div}.
The other direction 
follows from 
Proposition~\ref{prop:non-cfs-cubic-lower} below, in which we show that if $\G$ is not \cfs\ then $\Cayley \G$ contains a bi-infinite 
geodesic whose divergence is at least cubic.

\begin{prop}\label{prop:cfs-quadratic-upper}
If $\Gamma$ is \cfs\  then $\div \G (r) \preceq r^2$.
\end{prop}

\begin{proof}
By Example~\ref{ex:fourcycle}, 
a four-cycle in $\G$ corresponds to a subgroup $W'$
isomorphic to $D_\infty \times D_\infty$. 
Recall from Section~\ref{sec:davis} that for every $g \in W$, there is an isometrically embedded copy of the Cayley graph of ${W'}$ based at $g\in \Cayley{\G}$
By 
Theorem~\ref{thm:linear-div}, $\div{{D_\infty} \times {D_\infty}} (r)\simeq r$.  In fact it is not hard to see directly that given a pair of geodesic rays $\alpha$ and $\beta$ emanating from $e$ in $\Cayley{{D_\infty \times D_\infty}}$, there is an $r$-avoidant path connecting $\alpha(r)$ and $\beta(r)$ of length at most $2r$.  

{\it Step 1:}
We first address the case that $\G^4$ has a single component.  
Fix a $4$-cycle $\Theta$  in $\G$ and a geodesic ray $\alpha$ emanating from $e\in \Cayley{\G}$ whose support is contained in 
the set of vertex labels of $\Theta$.  Thus $\alpha$ lies   
in the copy of $\Cayley{\Theta}$ based at $e$.  
We show below that 
if $\beta$ is an arbitrary geodesic ray in $\Cayley \G$ emanating from $e$, then 
$\div{\alpha,\beta}(r)  \le Mr^2$
for every $r$, where $M = 2 \diam( \G^4)$.  
This proves the quadratic upper bound on $\div{\G}$, since it implies that if $\beta_1$ and $\beta_2$ are arbitrary 
geodesic rays based at $e$,  then 
$\div{\beta_1,\beta_2} (r) \le 2Mr^2$.

Now let $\beta$ be an arbitrary geodesic ray labelled $b_1b_2b_3\dots$ and emanating from $e$.  We first divide 
$\beta_{[0,r]}$
 into \emph{pieces} as follows, then carry out induction on the number of pieces.  Starting at $b_1$, choose the first piece to be the maximal word $b_1 \dots b_i$ 
 such that $\{b_1, b_2, \dots, b_i\}$ is contained in the set of vertex labels of a 
single 4-cycle of $\G$. Now repeat this procedure starting at $b_{i+1}$, and continue until $\beta_{[0,r]}$ is  
exhausted.

If $\beta_{[0,r]}$ consists of a single piece, then $b_1, \dots, b_r$ are among the vertices of a single 4-cycle $
\Theta'$ of $\G$.  
Since 
$\G^4$ is connected, it contains a path connecting the  fixed vertex $\Theta$ to $\Theta'$. Let 
$ \Theta =\Theta_1, \Theta_2, \dots,  \Theta_l=\Theta'$ be the vertices of $\G^4$ along this path.
For each $1\le i \le l-1$, since $\Theta_i$ and $\Theta_{i+1}$ are joined by an edge in $\G^4$, the intersection
$W_{\Theta_i} \cap W_{\Theta_{i+1}}$ is isomorphic to $W_{\Theta_i \cap \Theta_{i+1}} \cong C_2 \times D_\infty$.  

Recall from Remark~\ref{rmk:davis} that each $\Sigma_{\Theta_i}$ is an isometrically embedded Euclidean plane tesselated by big squares,  
and $\Sigma_{\Theta_i}$ and $\Sigma_{\Theta_{i+1}}$ intersect in an infinite band of big squares corresponding to a copy of 
$\Sigma_{\Theta_i \cap \Theta_{i+1}}$.  We proceed below by introducing geodesic rays $\nu_i$ based at $e$, where 
$\nu_i$ lies in $\Cayley{{\Theta_i \cap \Theta_{i+1}}} \subset \Sigma_{\Theta_i \cap \Theta_{i+1}}$ for 
$1 \le i \le l-1$.   Since successive geodesics in the sequence $\alpha=\nu_0, \nu_1, \dots \nu_{l-1}, \nu_l=\beta$
lie in a Euclidean plane, there are linear length avoidant paths between them, and concatenating these gives an avoidant path between $\alpha$ and $\beta$.

Let $\nu$ denote the geodesic in $ \Cayley{{C_2 \times{D_\infty}}}$ based at the identity and labelled 
$g_1g_2g_1g_2\dots$, where $g_1$ and $g_2$ are the generators of the $D_\infty$ factor.  
For $1 \leq i \leq l-1$, let $\nu_i$ denote the image of this geodesic in the copy of 
$\Cayley{{\Theta_i \cap \Theta_{i+1}}}$
based at $e$ in $\Cayley{\G}$ (for some identification of $g_1$ and $g_2$ with the Coxeter generators of the $D_\infty$ factor of $W_{\Theta_i} \cap W_{\Theta_{i+1}}$).   
Define $\nu_0 = \alpha$ and $\nu_{l} = \beta$, and observe that for $1\le i \le l$, the geodesics $\nu_{i-1}$ and 
$\nu_{i}$ are supported on a single $4$-cycle of $\G$, namely
$\Theta_i$. Thus $\nu_{i-1}(r)$ and $\nu_{i}(r)$ can be connected by an avoidant path of length at most $2r$ in the 
copy of $\Cayley{{\Theta_i}}$ based at $e$.
  Concatenating all of these paths, one obtains an $r$-avoidant path connecting $\alpha(r)$ and $\beta(r)$, with 
  length at most $2rl \le Mr$, since $l \le \mathrm{diam}( \G^4)$.

We now induct on the number of pieces of $\beta_{[0,r]}$ to show that 
$\avoid{\alpha(r)}{\beta(r)}$ is at most 
$Mr$ times the number of pieces.  
Suppose 
$\beta_{[0,r]}$ has $k+1$ pieces and is labelled by $w_1w_2 \dots w_kw_{k+1}$, where each $w_i$ is a piece.  Then
it is not hard to construct a word $w$ such that: 
\begin{enumerate}
\item 
the support of $w$ is contained in the support of the $4$-cycle corresponding to the piece~$w_k$;
\item  the word $w_kw$ is reduced; and 
\item $|w| = |w_{k+1}|$ (so that $|w_1w_2 \dots w_kw|=r$).
\end{enumerate}
It follows that the path $\mu$ emanating from $e$ labelled $w_1w_2\dots w_kw$ is a geodesic of length $r$ with 
$k$ pieces.  By the inductive hypothesis, there is an $r$-avoidant path connecting $\alpha(r)$ to $\mu(r)$ of length 
at most $Mkr$.  

Further, if $s=r-|w_{k+1}|$,  then $\beta_{[s,r]}$ and $\mu_{[s,r]}$  are supported on $4$-cycles 
$\Psi$ and $\Psi'$ respectively,
and $\beta(s)=\mu(s)$.   A more careful version of the construction for the base case yields an
$r$-avoidant path from $\mu(r)$ to $\beta(r)$, as follows.  As before, choose a path in $\G^4$ which visits the vertices
$\Psi = \Psi_1, \Psi_2 \dots, \Psi_m=\Psi'$, and for each $i$, choose a geodesic ray $\nu_i$ emanating from $
\beta(s)$ in the copy of $\Cayley{{W_{\Psi_i} \cap W_{\Psi_{i+1}}}}$ based at $\beta(s)$, but this time require $
\nu_i$ to have the additional property that $\beta_{[0,s]}$ concatenated with $\nu_i$ is a geodesic.  (This will be 
true for at least one of the two possibilities for $\nu_i$.) 
Now the construction from the base case (applied with basepoint $\beta(s)$ instead of $e$) yields a path that 
avoids not only the ball of radius $|w_{k+1}|$ based at $\beta(s)$, but also the ball of radius $r$ based at $e$. 
The length of this path is at most $M|w_{k+1}| \le  Mr$.
Concatenating the paths from $\alpha(r)$ to $\mu(r)$ and from $\mu(r)$ to $\beta(r)$, one has the desired $r$-avoidant path, with length clearly bounded above by 
$M(k+1)r$.

 Finally, since the total number of pieces is bounded above by $r$, the length of this avoidant path 
is bounded above by $Mr^2$.  

{\it Step 2:}
Now suppose that $\G$ is \cfs\ but $\G^4$ is not connected.  Then by Observation~\ref{obs:cfs}, there exists a subgraph $\Lambda$ of $\G$, such that $\Lambda^4$ is connected, and $\G$ is obtained from $\Lambda$ by adding edges 
(between vertices that are at least distance 3 apart in $\Lambda$).  Since the effect of adding edges is to add more commuting relations in the presentation, there is a natural quotient map $q: W_{\Lambda} \to W_{\G}$.  Hence if $\beta_1$ and $
\beta_2$ are arbitrary geodesic rays emanating from $e$ in $\Cayley{\G}$, they have pullbacks $\beta_1'$ and $
\beta_2'$ which are geodesic rays emanating from $e$ in $\Cayley{\Lambda}$.  

We claim that the pushforward of the $r$-avoidant path constructed 
in Step 1 between $\beta_1'(r)$ and $
\beta_2'(r)$ is $r$-avoidant in $\Cayley{\G}$. 
The path was constructed by concatenating several sub-paths, each of which was $r$-avoidant in a
sub-graph $\Cayley{\Psi}$, where $\Psi$ is a single four-cycle.  The 
claim follows from the observation that if $\Psi$ is an embedded four-cycle in $\Lambda$ then it is an 
embedded four-cycle in $\G$, and the composition of the induced map $q: \Cayley{\Lambda} \to \Cayley{\G}$ with the inclusion $\Cayley{\Psi} \hookrightarrow \Cayley{\Lambda}$ is actually an isometric 
embedding of $\Cayley{{\Psi}}$ into $\Cayley{\G}$.  
\end{proof}

\begin{remark}\label{rmk:thick-cfs}
Proposition~\ref{prop:cfs-quadratic-upper} is a special case of the upper bound on divergence given by Theorem~4.9 of~\cite{behrstock-drutu}.  To see this, suppose $\G$ is \cfs\ and let $\mathcal H$ be the collection of special subgroups of $W_\G$ generated by the embedded four-cycles in $\G$ which are the vertices of a component of $\G^4$ with full support.  Then it is easy to see that $W_{\G}$ is strongly algebraically thick of order at most 1 with respect to $\mathcal H$.  Hence by results in~\cite{behrstock-drutu}, the divergence of $W_\G$ is at most quadratic.   In fact, together with Proposition~\ref{prop:quadratic-lower} above, one sees that 
$W_\G$ is strongly algebraically thick of order exactly equal to 1 if and only if $\G$ is \cfs\ but not a join. 
\end{remark}

We now show that graphs which are not \cfs\ give rise to right-angled Coxeter groups with super-quadratic divergence.  
If $\G$ is not \cfs, then, in particular, it is not a join, and there is a word $w$ (of length $k$) and bi-infinite geodesic $\gamma$ in $\Cayley\G$ as 
described in Definition~\ref{def:gamma}.  We show that in this setting, the divergence of $\gamma$ is at least 
cubic.

\begin{prop}\label{prop:non-cfs-cubic-lower}
If $\G$ is not \cfs, and 
$\gamma$ is the bi-infinite geodesic in $\Cayley{\G}$ from Definition~\ref{def:gamma}, then $\div{\gamma}(r) \succeq r^3$. 
\end{prop}

\begin{proof}
Let $\eta$ be an arbitrary avoidant path from $\gamma(-nk)$ to $\gamma(nk)$.  We begin exactly as in the first paragraph of the proof of 
Proposition~\ref{prop:quadratic-lower} and define the sub-paths $\eta_i$ of $\eta$ as we did there.  
However, this time we use the fact that $\G$ is not \cfs\ to obtain a quadratic lower bound on $\ell(\eta_i)$.  
This is a consequence of the following lemma, which is proved separately below.  

\begin{lemma}\label{lem:vanKampen}
Suppose $\G$ is a graph that is not \cfs\ and 
$w$ is the word from Definition~\ref{def:gamma}.
Let $\alpha$ be an arbitrary geodesic ray emanating from $e$ that travels along $H_{a_1}$ 
and let $\beta$ be a path emanating from $e$ consisting of a geodesic segment labelled $w$ followed 
by an arbitrary geodesic ray emanating from $w$ that travels along $wH_{a_1}$. 
Then $\beta$ is a geodesic, and 
for any $r>2k$, 
$$
 \div{\alpha,\beta}(r) \ge \frac 1{16} r^2.$$
\end{lemma}

Note that $\gamma$ crosses the wall $w^iH_{{a_1}}$ at the edge $(w^i, w^i a_1)$.  Let $\nu_i$ denote the geodesic segment that connects $w^i$ to $g_i$ and runs along $w^iH_{{a_1}}$.  Let 
$\mu_i$ be the path emanating from $w^i$ consisting of the part of $\gamma$ between $w^i$
and $w^{i+1}$ concatenated with $\nu_{i+1}$.  Lemma~\ref{lem:vanKampen}, applied with basepoint $w^i$ instead of $e$, implies that $\mu_i$ is a geodesic, and that for 
$0 \le i \le n-2$, and $n>2$, 
$$\ell(\eta_i) \ge d^{\mathrm{av}}_{w^i}({g_i},{g_{i+1}}) \ge 
 d^{\mathrm{av}}_{w^i}({\nu_i(kn-ki)},{\mu_i(kn-ki)}) \ge
\frac {k^2} {16} (n-i)^2.$$
For the middle inequality above, we use the observation in the last paragraph of 
Section~\ref{sec:divergence}.
In conclusion, 
$$\ell(\eta) \ge  \sum_{i=0}^{n-2} \ell(\eta_i) \ge  \sum_{i=0}^{n-2} \frac {k^2}{16} (n-i)^2. $$
This is a cubic function of $n$.  
\end{proof}

\begin{proof}[Proof of Lemma~\ref{lem:vanKampen}]
We first show that $\beta$ is a geodesic ray.  Since $\beta_{[0,k]}$ (which is labelled by $w$) and 
$\beta_{[k,\infty]}$ are geodesics, the only way $\beta$ can fail to be a geodesic is if there is a wall
which intersects both of these.  Recall that $w=a_1 \dots a_k$, so that 
the walls crossed by $\beta_{[0,k]}$ are $\beta(i-1)H_{a_i}$ for $1 \le i \le k$,
where $\beta(0)=e$ and $\beta(i)=a_1 \dots a_i$. 
By construction, $a_i$ and $a_{i+1}$ don't commute for any $i$ (mod $k$), so it follows that 
these walls 
are pairwise disjoint, and are all disjoint from $wH_{a_1}$.  On the other hand every wall that intersects 
$\beta_{[k, \infty]}$ necessarily crosses $wH_{a_1}$, since $\beta_{[k, \infty]}$ runs along 
$wH_{a_1}$.  It follows that no wall can cross both $\beta_{[0,k]}$ and $\beta_{[k, \infty]}$.
Similarly, since $\alpha$ is a geodesic emanating from $e$ along the wall $H_{a_1}$, the same argument shows that no wall can cross both $\beta_{[0,k]}$ and $\alpha$, a fact that will be useful later in this proof.

To obtain a lower bound on $\div{\alpha,\beta}$, choose an arbitrary $r$-avoidant path $\eta$ between $\alpha(r)$ and $\beta(r)$.  Then one obtains a loop in $\Cayley{\G}$ by concatenating $\alpha_{[0,r]}$, followed by $\eta$, followed by $
\beta_{[0,r]}$ traversed in the negative direction.
There is a van Kampen diagram $D$ with boundary label equal to the word encountered 
along this loop.   Note that by construction, $\alpha_{[0,r]}$, $\beta_{[0,r]}$ and $\eta$ do not have 
 any common edges in $\Cayley{\G}$.
It follows that 
every edge of $\partial D$ is part of a $2$-cell of $D$, and that $D$ is homeomorphic to a disk.  We will abuse notation and use $\alpha$, $\beta$ and $\eta$ to denote the parts of $\partial D$ that are labelled by these paths.  

There is a label-preserving combinatorial map from $D$ to $\Sigma_{\G}$ with the cellulation by big squares. 
Under this map, 
edges and vertices of $D$ go to edges and vertices of $\Cayley{\G}$, which is the 1-skeleton of the cellulation by big squares.
We may assume that each $2$-cell in $D$ is a square, since any $2$-cell with boundary label of the form $s^2$ maps to an edge of $\Sigma_\G$ and can therefore be collapsed to an edge in $D$.
Thus the map takes each $2$-cell of $D$ homeomorphically to a big square of 
$\Sigma_{\G}$.   Further, if we metrise each square of $D$ as $[0,1] \times [0,1]$, then we can arrange that the restriction of this map to a square of $D$ is an isometry onto its image big square.  

We will work primarily with a cell structure on $D$ that is dual to the one just described.  We first define \emph{walls of $D$}, record some of their properties, and then use them to define the dual structure on $D$. 
The dual structure is then used to divide $D$ into \emph{strips}, and we will show that the length of a strip is a lower bound on the length of $\eta$.  We then finish the proof by inductively estimating the lengths of the strips.  The fact that $\G$ is not \cfs\ is used to show that the lengths of strips grow quadratically.

\subsection*{Walls of $D$}
Recall that each big square in $\Sigma_{\G}$ is subdivided into four small squares by a pair of (segments of)
walls which intersect at the centre of the big square.  For each square in $D$, we pull back this pair of segments to $D$, and label them with the type of the walls they came from.  
The types of the two wall-segments in a square of $D$ are necessarily distinct. 
Now suppose there are two squares in $D$ which share an edge $\epsilon$.  By construction, both squares contain a wall-segment that intersects $\epsilon$ at its midpoint, and these wall-segments must have the same label.  To see this, recall that the image of $\epsilon$ in 
$\Sigma_{\G}$ is the side of a big square, and the midpoint of such a side cannot be the point of intersection of a pair of walls.  Thus, starting at any wall-segment in a square of $D$, one can 
continue it through adjacent squares until it eventually meets $\partial D$.  We call a path constructed in this way a \emph{wall of $D$}, and the type of the wall is the type of any of its 
wall-segments.   
Walls are similar to corridors: 
if one ``fattens up'' a wall of type $a$ by taking the union of the squares containing its individual wall-segments, then one has an 
$a$-corridor of $D$.  

Two walls of $D$ intersect each other at most once;  they intersect only if their types commute and are distinct.   A wall of $D$ cannot intersect itself, as this would require there to be a square in $D$ in which both the wall-segments have the same type.  Thus each wall of $D$ is an embedded interval connecting a pair of points on $\partial D$.  We record the following observation for future use.

\begin{obs}\label{obs:D-walls}
Every wall of $D$ has at least one endpoint on $\eta$. 

To see this, recall from the first paragraph of this proof that in $\Sigma_{\G}$, and therefore in $D$, any 
wall intersecting $\beta_{[0,k]}$ is disjoint from both 
$\beta_{[k,r]}$ and $\alpha_{[0,r]}$.  
Thus any wall in $D$ with an endpoint on $\beta_{[0,k]}$ has its other endpoint on $\eta$.  
Now suppose there is a wall $P$ in $D$ with one endpoint on $\alpha$ and the other on $\beta_{[k,r]}$.  Then $P$ separates $D$,
putting $\beta_{[0,k]}$ and $\eta$ in different components.  This implies that every wall with an 
endpoint on $\beta_{[0,k]}$ intersects $P$.  
However, one of these walls has the same type as $P$, since 
$\beta_{[0,k]}$, which is labelled by $w$, has as its support the full vertex set of $\G$.  This is a contradiction.
\end{obs}

\begin{figure}
\begin{overpic}[scale=1.4,unit=1mm]%
{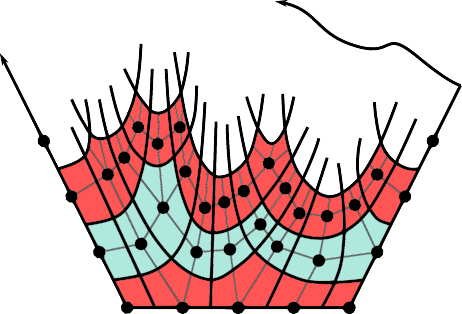}
\put(-1,46){$\alpha$}
\put(99,43){$\beta_{[k,r]}$}
\put(80,61){$\eta$}
\put(48,-4){\small $\beta_{[0,k]}$}
\put(27,-3){$e$}
\put(78,1){$S_0$}
\put(83,10){$S_1$}
\put(88,21){$S_2$}
\end{overpic}
\caption{The van Kampen diagram $D$. The light edges and bold dots are the $1$-cells and $0$-cells respectively, of the original cell structure on $D$.  The walls of $D$, which bound the dual $2$-cells, are shown in bold.  The strips are shaded. }\label{fig:strips}
\end{figure}
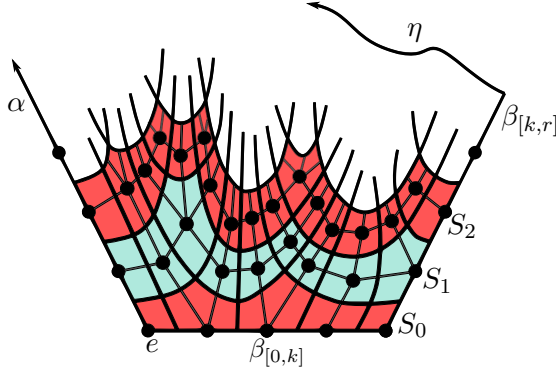

\subsection*{The dual cell structure on $D$}
We now define the dual structure on $D$.
Its $1$-skeleton is the union of the walls of $D$, together with $\partial D$; see Figure~\ref{fig:strips}.  
Thus the vertices are points of intersection of a pair of walls 
(i.e.~centres of squares in the original structure) or points of intersection of a wall with $\partial D$.  
Removing the vertices from the $1$-skeleton yields several components; the edges are the closures of these components.  
The 
$2$-cells are the closures of the components of the complement of the $1$-skeleton in $D$.  
We use the terms \emph{dual cells} and \emph{original cells} to distinguish between cells from the two structures on $D$. 
A dual cell is called a \emph{boundary cell} if it intersects $\partial D$.  Otherwise it is 
called an \emph{interior cell}.   
Since $\G$ is a 
triangle-free graph, it is easy to see that the boundary of any interior dual $2$-cell is a polygon with at least four sides.

\subsection*{Strips in $D$}

We now use the dual structure to define strips $S_i$ in $D$, for  
$0 \le i <(r-k)/2$. 

Define the $0$th strip $S_0$ to be the union of all the dual $2$-cells 
intersecting $\beta_{[0,k]}$.  
Define the \emph{top boundary} $B_0$ of $S_0$, by $B_0 = 
\partial S_0 \setminus \partial D$.   
Let  $\epsilon_{\alpha(j)}$ (respectively $\epsilon_{\beta(j)}$) 
denote the dual edge of $\partial D$ containing the original 
vertex
$\alpha(j)$ (respectively $\beta(j)$).  Observe that: 

\begin{enumerate}
\item $S_0$ is connected and 
consists of an ordered collection of dual 2-cells, each intersecting the previous one in a dual edge, and going from $\epsilon_{\alpha(0)}$ to $\epsilon_{\beta(k)}$.  

\item If $Q$ is a wall that forms part of $B_0$, then $S_0$ is contained in a single component of 
$D \setminus Q$.  

\item $B_0$ is connected, and all but the first and last dual edges of $B_0$ are interior edges. 
\end{enumerate}

Note that (1) 
follows from the fact that every edge of $\beta_{[0,k]}$ is part of a $2$-cell, and that $D$ is homeomorphic to a disk.  
If (2) fails, then $Q$ crosses $S_0$ and has an endpoint on $\beta_{[0,k]}$. 
On the other hand, since it is part of $B_0$, it contributes to the boundary of a boundary $2$-cell, and two of the 
boundary edges of this $2$-cell are parts of walls $P_1$ and $P_2$ which intersect $\beta_{[0,k]}$.  In order 
to intersect $S_0$, the wall $Q$ must cross either $P_1$ or $P_2$.  This is a contradiction, since by construction, no two walls with endpoints on $\beta_{[0,k]}$ intersect each other.
Finally, (3) follows from (1), together with the fact that the construction forces $B_0$ to consist solely of parts of walls.

Now suppose $S_{i-1}$ and its top boundary $B_{i-1}$ have been defined, with properties analogous to 
(1)-(3) above.  In particular, the $2$-cells of $S_{i-1}$ go from $\epsilon_{\alpha(i-1)}$ to 
$\epsilon_{\beta(k+i-1)}$.
Define $S_i$ to be the union of all the dual $2$-cells intersecting $B_{i-1}$.  
Then $S_i$ contains the dual 2-cells whose boundaries contain the edges
$\epsilon_{\alpha(i)}$ and $\epsilon_{\beta(k+i)}$.
 Define the top boundary $B_i$ to be 
$\partial S_i \setminus \{ B_{i-1},  \epsilon_{\alpha(i)}, \epsilon_{\beta(k+i)}\}$.   
We claim that if  $i <(r-k)/2$, then $S_i$ has properties analogous to (1)-(3) above.   

To see (1), note that property (1) for $S_{i-1}$ implies that $S_{i-1}$, and therefore $B_{i-1}$ separates $D$.   
Let $D_i$ be the closure of the component of $D\setminus B_{i-1}$ not containing $S_{i-1}$ (so that $\partial D_i$ consists of $B_{i-1}$ and a part of $\partial D$). 
By property (3) for $B_{i-1}$, all but the first and last dual edges of $B_{i-1}$ are interior edges of $D$,  so every edge of $B_{i-1}$ is part of a $2$-cell in $D_i$ and 
$D_i$ is homeomorphic to a disk.  It follows that $S_i$ is connected and
consists of an ordered collection of dual $2$-cells, each intersecting the previous one in a dual edge, going from $\epsilon_{\alpha(i)}$ to $\epsilon_{\beta(k+i)}$.

An argument involving intersections of walls similar to the $S_0$ case proves property (2) for $S_i$.  

Property (3) would fail for $B_i$ if one of the dual 2-cells of $S_i$ other than the first and the last is a boundary cell, as this would mean that $B_i$ contains part of $\alpha_{[i+1,r]}$, $\beta_{[k+i+1, r]}$, or $\eta$. (Note that 
 $\alpha_{[0, i-1]}$, and $\beta_{[0,k+i-1]}$
cannot be part of $B_i$ since $B_{i-1}$ separates $S_i$ from these parts of $\partial D$.)

 We first rule out 
$\alpha_{[i+1,r]}$ and $\beta_{[k+i+1, r]}$. 
Let $A_i$ denote the wall of $D$ with an 
endpoint 
at the intersection of $\epsilon_{\alpha(i)}$ and $\epsilon_{\alpha(i+1)}$. 
Note that $\alpha_{[i+1,r]}$ cannot cross $A_i$ by construction.  
Now $A_i$ is a part of $B_i$, so by property (2) for $S_i$ it separates $\alpha_{[i+1,r]}$ from
$S_{i}$.  This implies that $S_{i}$ cannot have any boundary cells intersecting $\alpha_{[i+1,r]}$. 
By the same argument, $S_i$ does not have any 
boundary cells intersecting $\beta_{[k+i+1,r]}$.

 The map from $D$ to $\Cayley \G$ takes each 
original vertex contained in a dual cell of $S_{i}$ 
into $B(e,k+2i) \subset \Cayley{\G}$.
To see this observe that each original vertex of $S_0$ is 
mapped into $B(e,k)$,
and for $j>0$, the image of an original vertex in $S_j$ is at most distance two from the image of the vertices of $S_{j-1}$.  So if $i< (r-k)/2$, then the original vertices of $S_{i}$ are mapped into $B(e,r-1)$, and therefore cannot be vertices of $\eta$, which is 
$r$-avoidant.  Thus $S_{i-1}$ does not have any boundary cells intersecting $\eta$.
This shows that all but the first and last 2-cells of $S_i$ are interior cells, which implies (3).


\subsection*{Lengths of strips}
Define the \emph{length} of $S_i$, denoted $\ell(S_i)$, to be the number of interior dual $2$-cells in it.  
\begin{claim} \label{clm:length-eta}
For $i < (r-k)/2$, we have $\ell(S_i) \le \ell(\eta)$.  
 \end{claim}

 \begin{proof}
Let $P$ be a wall of $D$ which is \emph{transverse} to $S_i$, meaning that it crosses
$S_i$ at least once, intersecting both $B_{i-1}$ and 
${B_i}$.   
We now show that $P$ crosses $S_i$ at most twice.   Further, the number of times $P$ crosses $S_i$ is equal to the number of endpoints of $P$ on $\eta$.  

Suppose $P$ crosses $S_i$ at least twice.  
Starting at the endpoint of $P$ on $\eta$ (guaranteed by Observation~\ref{obs:D-walls}), follow $P$ till its second crossing of $S_i$, and let $Q$ 
denote the top boundary wall at the second crossing.  
By property (2) for $S_i$,  we know that $S_i$ is contained in a single component of $D\setminus Q$. 
Thus, 
in order to cross $S_i$ again, $P$ would have to cross $Q$ a second time, 
which is impossible. 
So $P$ crosses $S_i$ at most twice. 

Now suppose the second endpoint of $P$ is on $\alpha$.  Since $Q$ can cross neither $P$ (a second time) nor $S_i$, it must also have an endpoint on $\alpha$.  
 This is a contradiction, 
since by construction, 
two walls with endpoints on $\alpha$ cannot intersect each other. 
 By the same argument, $P$ cannot have 
 an endpoint on $\beta$.  Thus, if $P$ crosses $S_i$ twice, it has two endpoints on $\eta$.
 If $P$ crosses $S_i$ exactly once, then Observation~\ref{obs:D-walls} and the fact that $S_i$ separates $D$ putting $\eta$ in a single component imply that $P$ has exactly one endpoint on $\eta$, completing the proof of the second statement above.  

Thus there is an injective map from the set of transverse intersections of walls with $S_i$ into 
 the set of walls crossed by $\eta$ in $\partial D$.
This proves the claim, as the number of 
such transverse intersections is 
$\ell(S_i)+1$, and the number of walls crossed by $\eta$ in $\partial D$ (and therefore in $\Sigma_{\G}$) is $\ell(\eta)$.
\end{proof}

\subsection*{Lower bounds}
We now inductively obtain lower bounds on the lengths of strips.  
Define an interior dual $2$-cell to be \emph{large} if its boundary has five or more sides.   

\begin{claim}\label{clm:large}
Every strip 
has at least one large $2$-cell.
\end{claim}

\begin{proof}
If not, then there is a strip $S_i$ built entirely out of squares.  There are two possibilities: 
either this strip consists of a single row of squares, or it consists a sequence of such rows of 
squares, with each such row connected to the next at right angles as in Figure~\ref{fig:square-strips}.
\begin{figure}\label{fig:square-strips}

\includegraphics{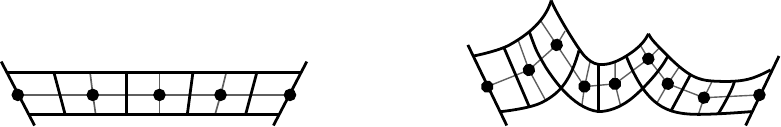}
\caption{A strip which has no large 2-cells is either a single row of squares (left) or a sequence of such rows (right).}
\end{figure}

Since two walls intersect only if the corresponding generators commute, it is possible to reconstruct a subgraph 
of $\G$ using $S_i$, as follows.  The vertices of this subgraph are the labels of the walls which meet $S_i$ (either transversely or as part of $B_i$ or $B_{i-1}$).  We add an edge between two such vertices of $\G$ whenever the corresponding walls intersect in $S_{i}$.  It is easy to see that a single row of squares reconstructs a join subgraph, while a sequence of 
rows of squares meeting at right angles reconstructs a \cfs\ subgraph.  
Every wall which has an endpoint on $\beta_{[0,k]}$ crosses $S_i$, since 
by Observation~\ref{obs:D-walls} its other endpoint is on $\eta$, and $S_i$ separates $\eta$ from $\beta_{[0,r]}$. 
Thus
the generators corresponding to walls with endpoints on $\beta_{[0,k]}$ are vertices of the subgraph constructed above.  But the support of 
$\beta_{[0,k]}$ is the entire vertex set of $\G$, so we obtain a \cfs\ subgraph of $\G$ which 
uses all the vertices of $\G$.  Then $\G$ itself is \cfs, by Observation~\ref{obs:cfs}.  This is a contradiction. 
\end{proof}

An interior dual $2$-cell in $S_i$ intersects $S_{i-1}$ in either an edge or a vertex.  Define the 
$2$-cell to be \emph{skew} if this intersection is a vertex.   Let $u_i$ denote the number of skew $2$-cells in 
$S_i$.

\begin{claim}\label{clm:ui}
For $1 \le i < (r-k)/2$, we have $u_i \ge i$ .
\end{claim}

\begin{proof}
To see that $u_1\ge 1$, note that $B_0$ cannot consist of a single wall, by Observation~\ref{obs:D-walls}.  
So it contains at least one pair of walls that intersect at a point and then pass through $S_1$, giving rise to a skew $2$-cell whose closure intersects $S_0$ in the point of intersection of the walls.

For the inductive step, observe that a skew $2$-cell in $S_{i-1}$ whose boundary is a 
$j$-gon gives rise to $j-3$ skew $2$-cells in $S_{i}$.  Since each interior dual $2$-cell has at least $4$ sides,  
$j-3 \ge 1$. 
Similarly, a non-skew large $2$-cell in $S_{i-1}$ whose boundary is a $j$-gon gives rise to $j-4 \ge 1$ skew $2$-cells in $S_{i}$. 
By Claim~\ref{clm:large}, every strip has at least one large $2$-cell.    
Now if one of the skew cells in $S_{i-1}$ is large, it gives rise to at least two skew cells in $S_{i}$, and we have 
$u_{i} \ge u_{i-1} +1$.  Otherwise there is a non-skew large cell in $S_{i-1}$, which gives rise to a skew cell in 
$S_{i}$ which does not come from a skew cell of $S_{i-1}$, and we have the same relation.  It follows that 
$u_i \ge i$ for $1 \le i < (r-k)/2$.  
\end{proof}

There is a map from the $2$-cells of $S_{i}$ to the $2$-cells of $S_{i-1}$ defined as follows.  The image of a skew 
$2$-cell $c$ is the unique $2$-cell in $S_{i-1}$ which shares a vertex with $c$. 
The image of a non-skew $2$-cell $c$ is the unique $2$-cell of $S_{i-1}$ which shares an edge with $c$.  
This is surjective by property (1) for $S_i$. 
The cardinality of the preimage is at least 1 for a non-skew $2$-cell, and at least $3$ for a skew $2$-cell of $S_{i-1}$.  Thus one has the relation $\ell(S_{i}) \ge \ell(S_{i-1}) + 2u_{i-1}$, since the length of a strip is the number of  interior $2$-cells in it.   Then, using Claim~\ref{clm:ui}, we have:

$$
\ell(S_{i}) \ge \ell(S_{i-1}) + 2 u_{i-1} \ge \cdots \ge  \sum_{j=1}^{i} 2 u_j \ge 2 \sum_{j=1}^{i} j \ge 
(i)(i+1) \ge i^2. 
$$
Finally, if $r>2k$, then $r/4 < (r-k)/2$, and by Claim~\ref{clm:length-eta}, we have 
$ \ell(\eta) \ge \ell(S_{r/4}) \ge \frac 1{16} r^2.$
\end{proof}

\section{Higher-degree polynomial divergence in right-angled Coxeter groups}\label{sec:higher-degree}

We now prove Theorem~\ref{thm:all degrees} of the introduction, by producing examples to show that the divergence of a 2-dimensional right-angled Coxeter group can be a polynomial of any degree. More precisely, if $\G_d$ is the sequence of graphs shown in 
Figure~\ref{fig:poly-graph} ($d \ge 1$) then we show that $\div{\G_d}(r) \simeq r^d$. 
We prove the upper and lower bounds on $\div{\G_d}(r)$ 
in Propositions~\ref{prop:poly-upper} and~\ref{prop:poly-lower} respectively.

\begin{figure}
\begin{overpic}[scale=1]%
{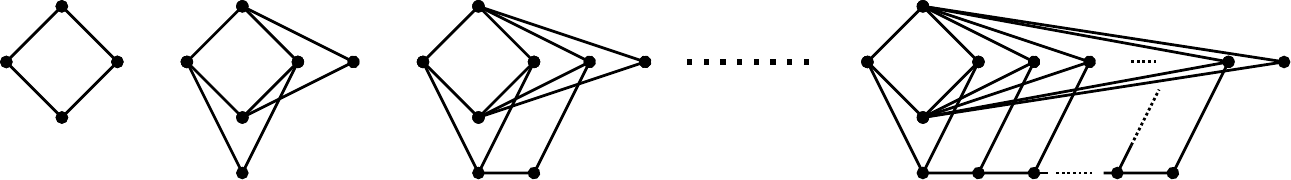}
\put(70.5,3){\tiny $b_0$}
\put(64.5, 9){\tiny $b_1$}
\put(70.5,-1.5){\tiny $b_2$}
\put(75,-1.5){\tiny $b_3$}
\put(79,-1.5){\tiny $b_4$}
\put(85,-1.5){\tiny $b_{d-1}$}
\put(91,-1.5){\tiny $b_d$}
\put(70.5,14.5){\tiny $a_0$}
\put(76.5,9){\tiny $a_1$}
\put(80.8,9){\tiny $a_2$}
\put(85,9){\tiny $a_3$}
\put(95,7.25){\tiny $a_{d-1}$}
\put(100.5,9){\tiny $a_{d}$}
\put(4, 16){\small $\G_1$}
\put(18,16){\small $\G_2$}
\put(36,16){\small $\G_3$}
\put(82, 16){\small $\G_d$}
\end{overpic}
\caption{}
\label{fig:poly-graph}
\end{figure}

\begin{prop}\label{prop:poly-upper}
$\div{\G_d} (r) \preceq r^d$.
\end{prop}

\begin{proof} 
Observe that the statement is true for $d=1$ and $2$, as
$\G_1$ is a join, and $\G_2$ is a \cfs\ graph.
We proceed by induction on $d$.  
Assume that there is a constant $C$ such that if $\mu$ and $\nu$ are arbitrary geodesic rays based at $e$ in 
$\Cayley{{\G_{d-1}}}$, then 
$\avoid{\mu(r)}{ \nu(r)} \le Cr^{d-1}$ for any $r$.

Now let $\alpha$ and $\beta$ be an arbitrary pair of geodesic rays based at $e$ in 
$\Cayley{{\G_d}}$. 
If neither of them crosses any walls of type $a_d$ or $b_d$, then they 
actually lie 
in the copy of $\Cayley{{\G_{d-1}}}$ based at $e$, 
and the induction hypothesis yields the desired avoidant path.

Thus we may assume that at least one of them, say $\alpha$, crosses a wall of type $a_d$ or $b_d$.   
Let $H_1, \dots, H_k$ be the ordered set of walls of type $a_d$ or $b_d$ that $\alpha$ crosses between $e$ and $\alpha(r)$, and let $x_i$ denote the type of $H_i$.  Then the label on 
$\alpha_{[0,r]}$ is $w_1x_1w_2x_2\dots w_kx_kw_{k+1}$, where each $w_i$ is a 
(possibly empty) word in the letters $a_0,a_1, \dots a_{d-1}, b_0, b_1, \dots b_{d-1}$, and each $x_i$ is 
$a_d$ or $b_d$.  
For $1 \le i \le k$, let $g_i$ denote the word $w_1x_1w_2x_2\dots w_i$.  
Then there exists a geodesic ray $\lambda_i$ emanating from $g_i$ with the following properties:
\begin{enumerate}
\item The path emanating from $e$ consisting of the segment labelled $g_i$ followed by $\lambda_i$ is a geodesic.

\item The geodesic $\lambda_i$ 
runs along $H_i$.  
(That is, the support of $\lambda_i$ is either $\{a_0, b_0\}$ or $\{a_{d-1}, b_{d-1}\}$, depending on whether 
$x_i$ is $a_d$ or $b_d$, respectively.)
\end{enumerate}
If $x_i=a_d$, the label of $\lambda_i$ must be of the form $a_0b_0a_0b_0\dots $ or $b_0a_0b_0a_0\dots$.  Choose the former if the projection of $g_i$ to the group $\langle a_0, b_0\rangle$ ends with $b_0$ 
and the latter otherwise.  This guarantees that there is no cancellation when $g_i$ is concatenated with the label of $\lambda_i$.  
The case $x_i=b_d$ is similar.

For $1 \le i \le k$, let $\nu_i$ be the geodesic ray emanating from $g_ix_i$ with the same label as $\lambda_i$.  (See Figure~\ref{fig:poly-upper}.)
For $0 \le i \le k-1$, let $\mu_i$ be the geodesic ray emanating from $g_ix_i$ (or $e$ when $i=0$) consisting of the segment with label $w_{i+1}$ followed by $\lambda_{i+1}$.    The choice of the 
$\lambda_{i}$ guarantees that these are geodesics.   Finally define $\mu_k$ to be the infinite part of $\alpha$ 
emanating from $g_kx_k$.

If $\beta$ does not cross any walls of type $a_d$ or $b_d$, then define $\mu_0'=\beta$.  Otherwise define $H_1', \dots H_l'$, as well as $x_i', g_i', u_i', \nu_i'$, and $\mu_i'$ analogous to the corresponding objects for $\alpha$.  

By construction, the supports of $\nu_i, \mu_i, \nu_i', \mu_i'$ are contained in $\{a_0,a_1, \dots a_{d-1}, b_0, b_1, \dots 
b_{d-1}\}$.  Thus there exist paths 
$\eta_i$ connecting $\nu_i(2r)$ and $\mu_i(2r)$ with length at most 
$C(2r)^{d-1}$,
which avoid a ball of 
radius $2r$ based at $g_i x_i$, and therefore avoid a ball of radius $r$ based at $e$.  Similarly, there are 
$r$-avoidant  paths $\eta_i'$ and $\eta_0$ connecting $\nu_i'(2r)$ and $\mu_i'(2r)$ and $\mu_0(2r)$ and 
$\mu_0'(2r)$ respectively, each with length at most $C(2r)^{d-1}$.

For each $i$,  the points 
$\mu_i(2r)$ and $\nu_{i+1}(2r)$ are connected by an edge, as are $\mu_i'(2r)$ and 
$\nu_{i+1}'(r)$.
Using these $k+l$ 
edges to connect $\eta_i$, $\eta_i'$ and $\eta_0$, 
one obtains an $r$-avoidant path between 
$\mu_k(2r)$ and $\mu_k'(2r)$.  Finally, $\eta$ is constructed by attaching the segment of $\alpha$ 
from $\alpha(r)$ to $\mu_k(2r)$ and the segment of
$\beta$ from $\beta(r)$ to $\mu_k'(2r)$, each with length at most $2r$.
Since $k$ and $l$ are at most $r$, we have:
$$\ell(\eta)\le 4r + k+l + (k+l+1)C(2r)^{d-1} \le 6r  + (2r+1)C2^{d-1}r^{d-1} \le C' r^d,$$
 where $C' = 6+2^{d+1}C$.
 \begin{figure}
\begin{overpic}[scale=.75]{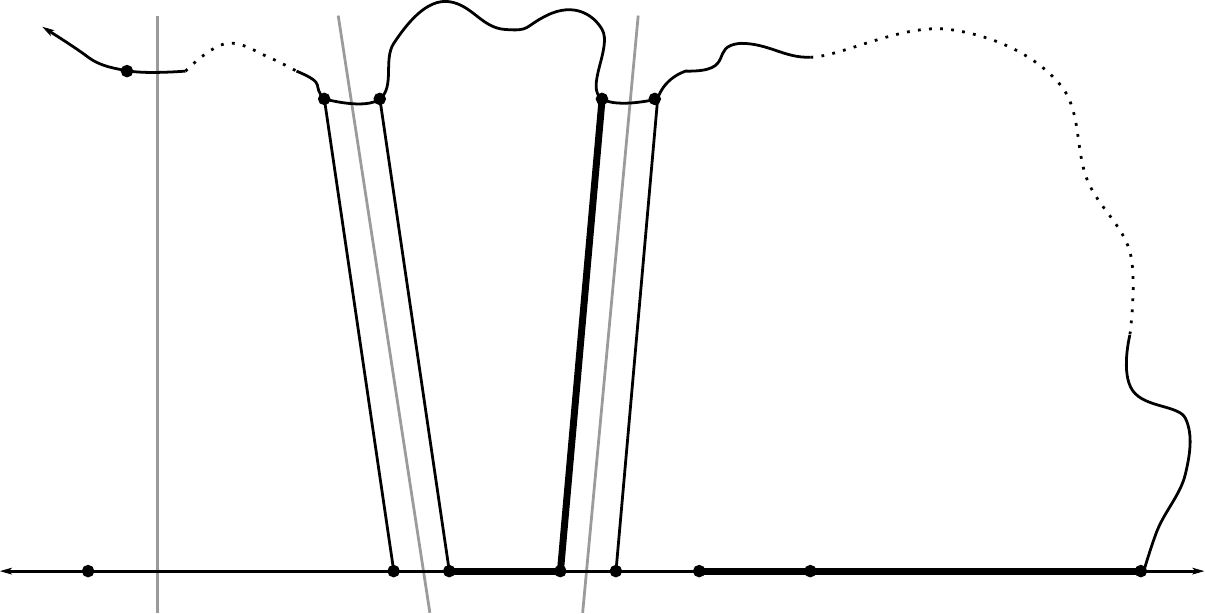}
\put(0,0){\small $\beta$}
\put(6,0){\small $e$}
\put(20,0){\small $\alpha$}
\put(8,47){\small $\eta_0$}
\put(8,30){\small $H_1$}
\put(27,51){\small $H_i$}
\put(53,51){\small $H_{i+1}$}
\put(40,52){\small $\eta_i$}
\put(25,30){\small $\lambda_i$}
\put(35,30){\small $\nu_i$}
\put(45,30){\small $\mu_i$}
\put(31,0){\small $g_i$}
\put(36,0){\small $g_ix_i$}
\put(66,0){\small $\alpha(r)$}
\put(91,0){\small $\mu_k(2r)$}
\put(99,8){\small $\eta_k$}
\put(82,0){\small $\mu_k$}
\put(56,0){\small $g_kx_k$}
\end{overpic}
\caption{Construction of the avoidant path $\eta$.  The geodesic rays $\mu_i$ and $\mu_k$ are 
shown in bold.}
\label{fig:poly-upper}
\end{figure}
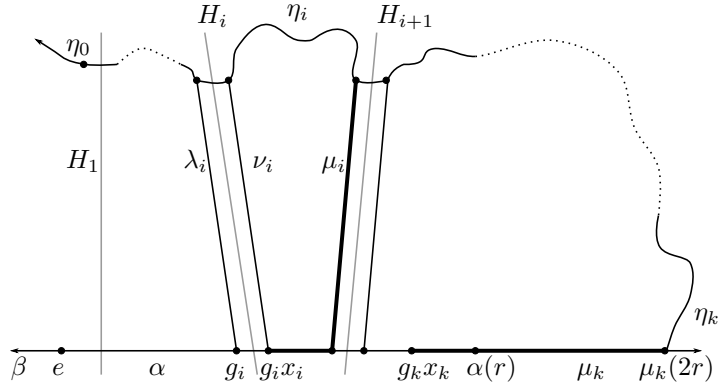
\end{proof}

\begin{remark}\label{rmk:thick}  This upper bound could also be obtained by arguments in \cite{behrstock-drutu}, as the group $W_d$ is strongly algebraically thick of order at most $d-1$.  To see this, for each $n \geq d \geq 1$ define a right-angled Coxeter group $W_{n,d}$ to be the special subgroup of $W_n$ generated by the set $\{a_0,a_1,\ldots,a_n, b_0,b_1,\ldots,b_d\}$.  Note that $W_{d,d} = W_d$.  Now $W_{n,2}$ is strongly algebraically thick of order at most $1$ since its defining graph is \cfs\ (see Remark~\ref{rmk:thick-cfs} above).  By induction on $d$, the group $W_{n,d}$ is strongly algebraically thick of order at most $d-1$ with respect to $\mathcal{H} = \{ W_{n,d-1}, \, b_d W_{n,d-1}b_d \}$.  Hence in particular, $W_d$ is strongly algebraically thick of order at most $d-1$.

\end{remark}

\begin{prop}\label{prop:poly-lower}
$\div{\G_d} (r) \succeq r^d$.
\end{prop}

\begin{proof}
We prove the lower bound  by producing a pair of geodesic rays in $\Cayley{{\G_d}}$ whose divergence is bounded below by a constant multiple of $r^{d}$.  This will follow from a more general statement about the divergence of certain pairs of geodesics in  $\Cayley{{\G_{d+2}}}$.

 For $1 \le n \le d$, let $\alpha_n$ and $\beta_n$ 
  be any geodesic rays in $\Cayley{{\G_{d+2}}}$ satisfying the following conditions:
 \begin{enumerate}
 \item $\alpha_n$ emanates from $e$ and  travels along $H_{b_{n+1}}$; and  
 \item $\beta_n$ emanates from $e$ and  travels along one of $H_{ a_{n}}, H_{b_{n}}$, or $H_{b_{n+2}}$. (Note that $\{a_n, b_n, b_{n+2}\}$ is exactly the set of types of walls which can intersect $H_{b_{n+1}}$.)
\end{enumerate}
Then 
we show below that  
\begin{equation}\label{eq:alphanbetan}
\avoid{\alpha_n(r)}{\beta_n(r)}  \ge\frac 1 {2^{n(n+1)}} r^{n}.
\end{equation}

When $n=d$, one can take $\alpha_d$ to be the geodesic ray based at $e$ with label 
$b_da_db_da_d \dots$, as this travels along $H_{{b_{d+1}}}$, and $\beta_d$ to be the geodesic 
ray based at $e$ with label $b_{d-1}a_{d-1}b_{d-1}a_{d-1}\dots$, as this travels along $H_{{b_d}}
$.  Observe that these geodesics are actually in the copy of $\Cayley{{\G_d}}$ based at $e$.  Any 
avoidant path between $\alpha_d(r)$ and $\beta_d(r)$
in $\Cayley{{\G_d}}$ remains avoidant under the isometric inclusion $\Cayley{{\G_d}} \hookrightarrow \Cayley{{\G_{d+2}}}$, and therefore 
has length bounded below by 
$(1/2^{d(d+1)})r^d$, 
by (\ref{eq:alphanbetan}). This completes the proof of the proposition.

We establish (\ref{eq:alphanbetan})
by proving the following equivalent statement by induction on $k$:
 for all $1 \le k \le d$ and all $k \le n \le d$,  if $\alpha_n$ and $\beta_n$ satisfy the conditions
(1) and (2) respectively, then 
$\avoid{\alpha_n(r)}{\beta_n(r)} \ge (1/2^{k(k+1)})r^{k}$.  

Observe that for any $n$, if $\alpha_n$ and $\beta_n$ are chosen as above, then 
$\alpha_n$ 
concatenated with $\beta_n$ at $e$ is a bi-infinite geodesic,  since $\alpha_n
$ and $\beta_n$ have disjoint supports, regardless of the type of wall along which $\beta_n$ 
travels. Thus any avoidant path between $\alpha_n(r)$ and $\beta_n(r)$ must cross the $2r$ walls crossed 
by this bi-infinite geodesic.  
This proves the case $k=1$, as $\avoid{\alpha_n(r)}{\beta_n(r)} \ge 2r > (1/4)r$ for all $n \ge 1$.

Now suppose $n \ge k+1$, and let $\eta$ be an avoidant path connecting $\beta_n(r)$ to $\alpha_n(r)$.  
 Focus on the $r/2$ walls that $\alpha_n$ crosses between $\alpha_n(0)$ and $\alpha_n(r/2)$.  
 Each of these is of type $ a_{n}, b_{n}$, or $b_{n+2}$.  
Since two consecutive walls cannot be of the same type, at most half of these walls are of type 
$a_{n}$.  Thus, in this range, $\alpha_n$ (and hence $\eta$) crosses at least $r/4$ walls of type $b_{\ast}$, where the 
subscript is either ${n}$ or ${n+2}$.  Call them $H_1, \dots, H_l$, where $l \ge r/4$.  
Let $(g_i, g_ib_\ast)$ be the edge where $\alpha_n$ crosses $H_i$ and let $(h_i, h_ib_\ast)$ 
be the first edge where $\eta$ crosses $H_i$, going from $\beta_n(r)$ to $\alpha_n(r)$.
Let $\mu_i$ denote the unique geodesic connecting $g_i$ to $h_i$ that travels along $H_i$.   
(See Figure~\ref{fig:poly-lower}.) 
Define $\mu_0=\beta_n$ and $h_0=\beta_n(r)$.

\begin{figure}
\begin{overpic}[scale=1,unit=1mm]%
{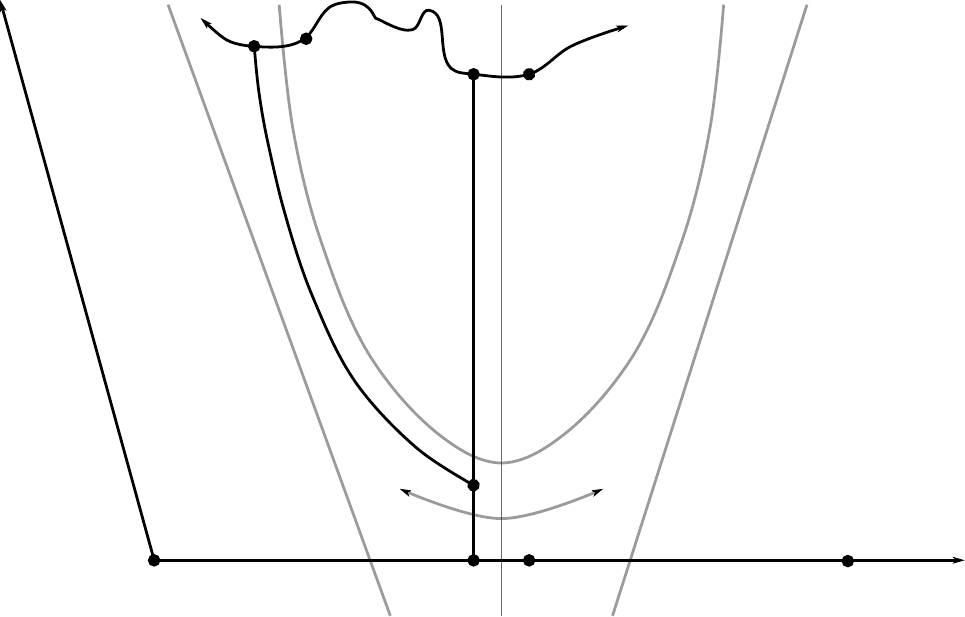}
\put(53,27){\small $H_i$}
\put(67,45){\small $H_i'$}
\put(22,27){\small $H_{i-1}$}
\put(73,27){\small$H_{i+1}$}
\put(34,22){\small $\nu_i$}
\put(12,2){\small $e$}
\put(45,40){\small $\mu_i$}
\put(4, 32){\small $\beta_n$}
\put(22,3){\small $\alpha_n$}
\put(86,2){\footnotesize $\alpha_n(\tiny r/2)$}
\put(47,2){\small $g_i$}
\put(45,53){\small $h_i$}
\put(23,57){\small $p_i$}
\put(55,2){\small $g_ib_\ast$}
\put(56,53){\small $ h_ib_\ast$}
\put(32,57){\small $p_ix_i$}
\put(40,63){\small $\eta_i$}
\end{overpic}
\caption{Construction of $\mu_i$ and $\nu_i$.  Here $\mu_i(0)=g_i$ and $\nu_i(0)=\mu_i(1)$. }
\label{fig:poly-lower}
\end{figure}
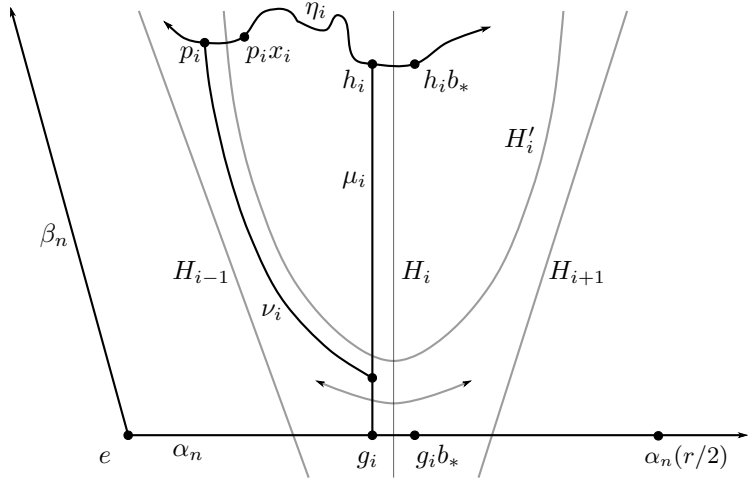

For $1\le i \le l$, let $H_i'$ denote the second wall crossed by $\mu_i$ starting at $g_i$.  
Note that $H_i'$ intersects $H_i$, and therefore cannot intersect $\alpha_n$, since no two walls crossed by $\alpha_n$ 
intersect.  
We claim that $H_i'$ also does not intersect $\mu_{i-1}$.  
The support of $\mu_i$ is contained in either $\{a_{n-1}, b_{n-1}, 
b_{n+1}\}$ 
or $\{a_{n+1}, b_{n+1}, b_{n+3}\}$ depending on the type of $H_i$.  
If the first wall crossed by $\mu_i$ doesn't intersect $\mu_{i-1}$, then $H_i'$ can't either, as 
$H_i'$ is separated from $\mu_{i-1}$ by the first wall.  Otherwise, the first wall crossed by $\mu_i$ has to be of type $b_{n+1}$, which 
means that $H_i'$ is not of type $b_{n+1}$.  If the types of $H_i$ and $H_{i- 1}$ are different, 
then the type of $H_i'$ is not in the support of $\mu_{i - 1}$, so $H_i'$ cannot intersect $\mu_{i - 1}$.  Finally, if the 
types of $H_i$ and $H_{i- 1}$ are the same, then they must be separated by a wall of type 
$a_n$, since $\alpha_n$ can't cross two consecutive walls of the same type.  Now $H_i'$ can't intersect 
this wall, since it is not of type $a_0, b_0$ or $b_{n+1}$.  So $H_i'$ can't intersect $\mu_{i - 1}$ 
either.

It follows that for $1 \le i \le l$, the wall $H_i'$ separates the points $h_{i-1}$ and $h_i$, since 
the path formed by concatenating $\mu_{i-1}$,
the part of $\alpha_n$ between $g_{i-1}$ and $g_i$, and $\mu_{i}$ crosses $H_i'$ exactly once.  Now 
$\eta$ contains a sub-path
connecting $h_{i-1}$ and $h_i$, so $\eta$ must cross $H_i'$. Let $(p_i, p_ix_i)$ be the first edge along which it crosses $H_i'$,
where $x_i$ is the type of $H_i'$.  Let
 $\eta_i$ denote the part of $\eta$ between $p_i$ and $h_i$, and 
let $\nu_i$ denote the unique geodesic 
connecting $\mu_i(1)$ to $p_i$, which travels along $H_i'$.

Observe that $\mu_i$ is a geodesic that travels along a wall of type $b_n$ or $b_{n+2}$, and $
\nu_i$ is a 
geodesic that travels along a wall that intersects it.  This means that the pair $\mu_i$ and $\nu_i$ 
is either of the form $\alpha_{n-1}$ and $\beta_{n-1}$ or 
$\alpha_{n+1}$ and $\beta_{n+1}$ (if we allow the geodesics to emanate from $\mu(1)$ instead of 
$e$).  
 Since $n-1 \ge k$, 
the inductive hypothesis applies, 
and we have that $d^{\mathrm{av}}_{\mu(1)}(\mu_i(s+1), \nu_i(s)) \ge 
(1/2^{k(k+1)})s^k$ for all $s$.
Since we restricted to the walls $H_1, \dots, H_l$ crossed by $\alpha_n$ between $e$ and $
\alpha_n(r/2)$, we know that $|g_i| \le r/2$.  On the other hand, since $h_i$ and $p_i$ are 
$r$-avoidant, the lengths of   
$\mu_i$ and  $\nu_i$ are at least $r/4$.  
By the observation and the end of Section~\ref{sec:divergence},  
$\ell(\eta_i) \ge (1/2^{k(k+1)}) (r/4)^k$ for all $i$.
  So, since $l \ge 4$,  we have 
$$
\ell(\eta) \ge \sum_{i=1}^{l} \ell(\eta_i) \ge 
l \left(\frac 1 {2^{k(k+1)}}\right) \left(\frac r4\right)^k \ge 
\left(\frac r4\right) \left(\frac {r^k}{2^{k(k+1)+k}}\right) =
 \frac 1 {2^{(k+1)(k+2)}} r^{k+1}
$$
as required.
\end{proof}

\appendix

\section{Relationship with examples of Macura}\label{app:macura}

In this appendix we discuss the relationship between our constructions of $\CAT(0)$ groups with divergence polynomial of any degree, and those of Macura~\cite{macura}.   

For $d \geq 2$, we denote by $G_d$ the group constructed in~\cite{macura} with presentation
\[ G_d = \langle a_0, a_1, \ldots, a_d \mid a_0a_1 = a_1a_0 \mbox{ and } a_i^{-1}a_0 a_{i} = a_{i-1}  \mbox{ for $2 \leq i \leq d$}\rangle. \] 
Let $X_d$ be the presentation $2$-complex for this presentation of $G_d$.  Then $X_d$ has a single vertex $v$, $d+1$ oriented edges labeled by $a_0,a_1,\ldots,a_d$, and $d$ squares with boundary labels $a_0a_1a_0^{-1}a_1^{-1}$ and $a_i^{-1}a_0 a_{i} a_{i-1}^{-1} $ for $2 \leq i \leq d$. Equip $X_d$ with the metric such that each square is a unit Euclidean square.  Then the universal cover $\widetilde{X_d}$ is a $\CAT(0)$ square complex, in which the link of every vertex is the graph $\G_d$  from Figure~\ref{fig:poly-graph} above.  
The link of any vertex in the Davis complex for $W_d= W_{\G_d}$ with the cellulation by big squares is also 
$\G_d$.
This observation is why we consider the relationship between $G_d$ and $W_d$.  To avoid confusion with Macura's notation, in this section we  relabel the vertices of $\G_d$ by $s_{i^+} = a_i$ and $s_{i^-} = b_i$ for $0 \leq i \leq d$.  

We would like to use covering theory to investigate common finite index subgroups of $G_d$ and $W_d$.  Any finite index subgroup of $G_d$ is the fundamental group of a finite square complex $Q$ such that there is a combinatorial covering map $\Psi:Q \to X_d$.  However since the group $W_d$  has torsion a more sophisticated covering theory is needed; as we explain below, its finite index subgroups correspond to finite-sheeted covers of complexes of groups.  We first recall some background on complexes of groups in Section~\ref{sec:complexes of groups}.  We then use this theory to show in Section~\ref{sec:commens} that $W_2$ and $G_2$ are commensurable, and to explain in Section~\ref{sec:d>2} why for $d > 2$ the covering-theoretic arguments used to find a common finite index subgroup in the case $d = 2$ cannot be applied.  

\subsection{Complexes of groups}\label{sec:complexes of groups}

We adapt the theory of complexes of groups and their coverings to our situation.  The general theory and details can be found in~\cite[Chapter III.$\mathcal{C}$]{bridson-haefliger}.  Throughout this section, $W = W_\G$ is a right-angled Coxeter group with $\G$ satisfying the hypotheses of Theorem~\ref{thm:linear and quadratic}, and $\Sigma$ is the associated Davis complex.

Let $Y$ be a square complex.  Assume that the edges of $Y$ may be oriented so that: 
\begin{itemize} \item[$(*)$] for each square of $Y$, if the positively oriented edge labels of this square are $a$, $b$, $a'$ and $b'$, then $b'a'a^{-1}b^{-1}$ is the boundary label. \end{itemize} For an oriented edge $e$ of $Y$, we denote by $i(e)$ its initial vertex and by $t(e)$ its terminal vertex.

\begin{examples}\label{ex:orient}  Two important examples of square complexes with edge orientations satisfying $(*)$ are the following.
\begin{enumerate}
\item Let $Y$ be the chamber $K$ with the cellulation by small squares.  For all pairs of spherical subsets $T' \subsetneq T$, we orient the edge of $Y$ connecting the vertices $\s_{T'}$ and $\s_T$ so that this edge has initial vertex $\s_{T'}$ and terminal vertex $\s_{T}$.  Note that every edge incident to $\s_\emptyset$ has initial vertex $\s_\emptyset$.
\item Similarly, if $Y = \Sigma$ with the cellulation by small squares, then the edges of $\Sigma$ may be oriented by inclusion of type.  
\end{enumerate}
\end{examples}

\noindent Now suppose that $Y$ and $Z$ are square complexes with edge orientations satisfying $(*)$.

\begin{definition}\label{d:morphism_scwols}  A \emph{nondegenerate morphism} $f:Y \to Z$ is a map taking vertices to vertices and edges to edges, such that:
\begin{enumerate}
\item for each square of $Y$, the restriction of $f$ to this square is a bijection onto a square of $Z$; and
\item for each vertex $\s$ of $Y$, the restriction of $f$ to the set of edges with initial vertex $\s$ is a bijection onto the set of edges of $Z$ with initial vertex $f(\s)$.
\end{enumerate}
\end{definition}

\noindent For example, if $Y = \Sigma$ and $Z = K$ with the orientations specified in Examples~\ref{ex:orient} above, then the quotient map $f:Y \to Z$ induced by the action of $W$ on $\Sigma$ is a nondegenerate morphism.   

\begin{definition}\label{d:cx of gps}  Let $Y$ be a square complex with edge orientations satisfying $(*)$.  
A \emph{complex of groups} $\mathcal{G}(Y)=(G_\sigma, \psi_e)$ over $Y$ consists of: \begin{enumerate} \item a group $G_\sigma$ for each
vertex $\sigma$ of $Y$, called the \emph{local group} at $\sigma$; and
\item a monomorphism $\psi_e: G_{i(e)}\rightarrow G_{t(e)}$ along each edge $e$ of $Y$.
\end{enumerate}
\end{definition}

\noindent A complex of groups is \emph{trivial} if each local group is trivial.  

\begin{example}\label{ex:WK}  We construct a canonical complex of groups $\mathcal{W}(K)$ over $K$ as follows.  For each spherical subset $T \in \cS$, the local group at the vertex $\sigma_T$ is the special subgroup $W_T$.  All monomorphisms along edges are inclusions. \end{example} 

The complex of groups $\mathcal{W}(K)$ in Example~\ref{ex:WK} is canonically induced by the action of $W$ on $\Sigma$.  More generally, if $G$ is a subgroup of $W$ then the action of $G$ on $\Sigma$ induces a complex of groups $\mathcal{G}(Y)$ over $Y = G \backslash \Sigma$, such that for each vertex $\s$ of $Y$, the $G$-stabiliser of each lift $\overline\s$ of $\s$ in $\Sigma$ is a conjugate of the local group $G_\s$ of $\mathcal{G}(Y)$.  A complex of groups is \emph{developable} if it is isomorphic to a complex of groups induced by a group action.  Complexes of groups, unlike graphs of groups, are not in general developable.  

See~\cite{bridson-haefliger} for the definition of the \emph{fundamental group} $\pi_1(\mathcal{G}(Y))$ and \emph{universal cover} of a (developable) complex of groups $\mathcal{G}(Y)$.  The universal cover of $\mathcal{G}(Y)$ is a connected, simply-connected square complex $X$, equipped with an action of $G = \pi_1(\mathcal{G}(Y))$ so that $Y = G \backslash X$.    

\begin{examples}
\hfill
\begin{enumerate}
\item The complex of groups $\mathcal{W}(K)$ has fundamental group $W$ and universal cover $\Sigma$.   
 \item Let $\mathcal{G}(Y)$ be the trivial complex of groups over a square complex $Y$.  Then $\pi_1(\mathcal{G}(Y))$ is the (topological) fundamental group of $Y$, and $\pi_1(\mathcal{G}(Y))$ acts freely on the universal cover of $\mathcal{G}(Y)$.
\end{enumerate}
\end{examples}

\noindent If a complex of groups $\mathcal{G}(Y)$ is developable, then each local group $G_\s$ naturally embeds in the fundamental group $\pi_1(\mathcal{G}(Y))$.

We now discuss coverings of complexes of groups.  We will only need to construct coverings $\mathcal{G}(Y) \to \mathcal{W}(K)$ where $\mathcal{G}(Y)$ is a trivial complex of groups, and so do not give the general definition, which is considerably more complicated.

\begin{definition}\label{d:covering} Let $Y$ be a square complex with edge orientations satisfying $(*)$.  Let $\mathcal{G}(Y)$ be the trivial complex of groups over $Y$.  A \emph{covering} of complexes of groups $\Phi: \mathcal{G}(Y) \to \mathcal{W}(K)$ consists of:
\begin{enumerate}
\item a nondegenerate morphism $f:Y \to K$; and 
\item \label{i:bijection} for each edge $e$ of $Y$, with $f(t(e)) = \s_T$, an element $\phi(e) \in W_{T}$; \end{enumerate} such that for each vertex $\sigma$ of $Y$ and each edge $e'$ of $K$, with $t(e') = f(\sigma) = \s_T$ and $i(e') = \s_{T'}$, the map 
\[  \Phi_{\s/e'}: \{ e \in f^{-1}(e') \mid t(e)=\sigma\}  \to W_T/W_{T'}\] induced by $e \mapsto
\phi(e)$ is a bijection.
 \end{definition}

\noindent Observe that if $e'$ is an edge of $K$ with $t(e') = \s_T$ and $i(e') = \s_{T'}$, then $|T| = |T'| + 1$, hence if $T = T' \cup \{t\}$ we have $W_T/W_{T'} \cong \la t \ra \cong C_2$.  So the condition in Definition~\ref{d:covering} that $\Phi_{\s/e'}$ is a bijection is equivalent to the condition that the set $\{ e \in f^{-1}(e') \mid t(e)=\sigma\}$ has two elements say $e_1$ and $e_2$, such that without loss of generality $\phi(e_1) \in W_{T'}$ and $\phi(e_2) \in tW_{T'}$.  In particular, it suffices to put $\phi(e_1) = 1$ and $\phi(e_2) = t$.  A covering $\Phi: \mathcal{G}(Y) \to \mathcal{W}(K)$ as in Definition~\ref{d:covering} is \emph{finite-sheeted} if $Y$ is finite.

The following result is a special case of a general theorem on functoriality of coverings of complexes of groups.  The general result is implicit in~\cite{bridson-haefliger}, and stated and proved explicitly in~\cite{lim-thomas}.

\begin{thm}\label{cor:W_d cover}  Let $K_d$ be the chamber for $W_d$, cellulated by small squares.   Let $\mathcal{W}(K_d)$ be the complex of groups over $K_d$ described in Example~\ref{ex:WK} above, with fundamental group $W_d$.    Then any subgroup of $W_d$ is the fundamental group of a complex of groups $\mathcal{G}(Y')$ (not necessarily trivial) over a square complex $Y'$, such that there is a covering of complexes of groups $\Phi:\mathcal{G}(Y') \to \mathcal{W}(K_d)$.  Moreover, a subgroup of $W_d$ has finite index if and only if it is the fundamental group of $\mathcal{G}(Y')$ such that there is a finite-sheeted covering  $\Phi:\mathcal{G}(Y') \to \mathcal{W}(K_d)$.
\end{thm}

\subsection{Commensurability in the case $d = 2$}\label{sec:commens}

We now use covering theory to prove the following. 

\begin{prop}\label{prop:commens}  The groups $G_2$ and $W_2$ are commensurable.
\end{prop}

\begin{proof}  Denote by $Z_2$ the first square subdivision of the presentation $2$-complex $X_2$.   We will construct a finite square complex $Y$ such that:
\begin{enumerate} \item there is a combinatorial covering map $\Psi:Y \to Z_2$; and \item there is a covering of complexes of groups $\Phi:\mathcal{G}(Y) \to \mathcal{W}(K_2)$, where $\mathcal{G}(Y)$ is the trivial complex of groups over $Y$.  \end{enumerate} Since $\mathcal{G}(Y)$ is the trivial complex of groups, the fundamental group of $\mathcal{G}(Y)$ is just the (topological) fundamental group of $Y$.  It follows that $G_2$ and $W_2$ are commensurable.

The square complex $Y$ will be the first square subdivision of the square complex $Q$ constructed below.  We will show that there is an $8$-sheeted combinatorial covering map from $Q$ to $X_2$, which implies~(1).  See Figure~\ref{fig:Q}; the complex $Q$ is obtained by carrying out some further edge identifications on this square complex.  

\begin{figure}[ht]
\begin{center}
\begin{overpic}[scale=0.5]{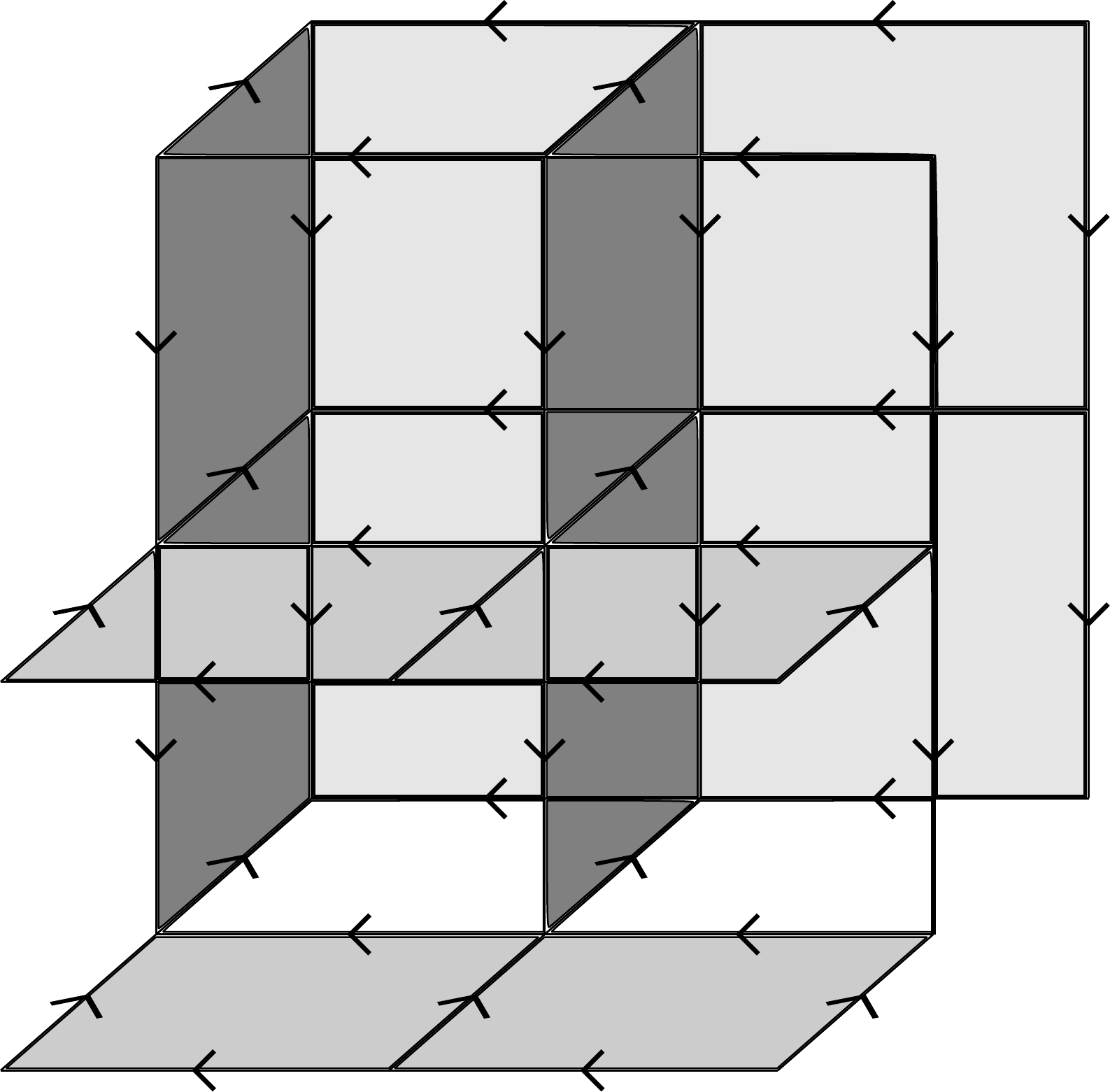}
\put(-4,4){$v_1$}
\put(32,4){$v_2$}
\put(67,4){$v_1$}
\put(-4,39){$v_4$}
\put(31,39){$v_3$}
\put(67,39){$v_4$}
\put(9,16){$v_5$}
\put(44,16){$v_6$}
\put(79,16){$v_5$}
\put(9,51){$v_8$}
\put(44,51){$v_7$}
\put(79,51){$v_8$}
\put(10,86){$v_5$}
\put(45,86){$v_6$}
\put(81,86){$v_5$}
\put(23,28){$v_1$}
\put(58,28){$v_2$}
\put(93,28){$v_1$}
\put(23,63){$v_4$}
\put(58,63){$v_3$}
\put(93,63){$v_4$}
\put(25,98){$v_1$}
\put(60,98){$v_2$}
\put(96,98){$v_1$}
\end{overpic}
\caption{The square complex $Q$, with vertices labelled and edges oriented, prior to some edge identifications.  All squares except for the four squares with vertex set $\{v_5,v_6,v_7,v_8\}$ are shaded.}\label{fig:Q}
\end{center}
\end{figure}
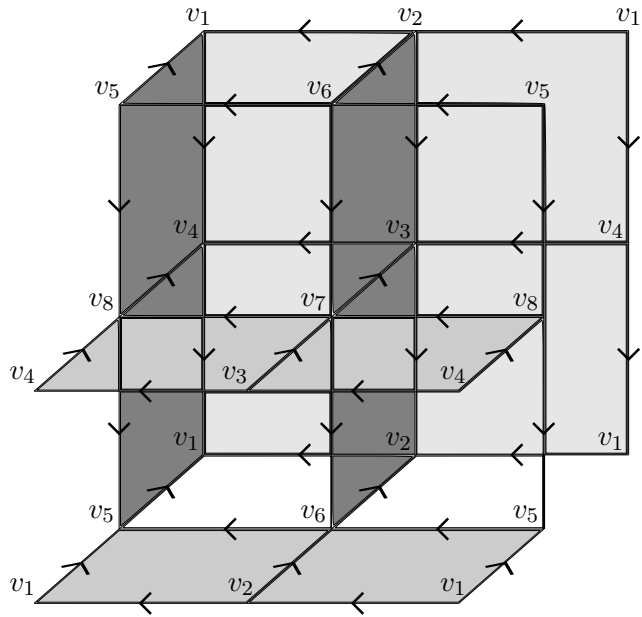

The complex $Q$ has $8$ vertices $v_1,\ldots,v_8$, which each get mapped to the vertex $v$ of $X_2$.  There are $24$ oriented edges of $Q$ which form three families as follows.  Here, $a_{i,j} = (v_k,v_l)$ means that the edge $a_{i,j}$ is the unique edge of $Q$ with initial vertex $v_k$ and terminal vertex $v_l$.
\begin{enumerate}
\item  The following $8$ edges get mapped to the edge $a_0$ of $X_2$:  $a_{0,1} = (v_1,v_2)$, $a_{0,2} = (v_2,v_1)$, $a_{0,3} = (v_4,v_3)$, $a_{0,4} = (v_3,v_4)$, $a_{0,5} = (v_6,v_7)$, $a_{0,6} = (v_7,v_6)$, $a_{0,7} = (v_5,v_8)$, $a_{0,8} = (v_8,v_5)$.  
\item  The following $8$ edges get mapped to the edge $a_1$ of $X_2$: $a_{1,1} = (v_1,v_4)$, $a_{1,2} = (v_4,v_1)$, $a_{1,3} = (v_2,v_3)$, $a_{1,4} = (v_3,v_2)$, $a_{1,5} = (v_5,v_6)$, $a_{1,6} = (v_6,v_5)$, $a_{1,7} = (v_8,v_7)$, $a_{1,8} = (v_7,v_8).$  
\item The following $8$ edges get mapped to the edge $a_2$ of $X_2$:  $a_{2,1} = (v_1,v_5)$, $a_{2,2} = (v_5,v_1)$, $a_{2,3} = (v_4,v_8)$, $a_{2,4} = (v_8,v_4)$, $a_{2,5} = (v_3,v_7)$, $a_{2,6} = (v_7,v_3)$, $a_{2,7} = (v_2,v_6)$, $a_{2,8} = (v_6,v_2)$.  
\end{enumerate}
We then attach $16$ squares along the following edge labels, forming two families as follows.  
\begin{enumerate}
\item The following $8$ squares get mapped to the square of $X_2$ attached along $a_0a_1a_0^{-1}a_1^{-1}$:  $$a_{0,1}a_{1,3}a_{0,3}^{-1}a_{1,1}^{-1}, \quad a_{0,2}a_{1,1}a_{0,4}^{-1}a_{1,3}^{-1},  \quad a_{0,3}a_{1,4}a_{0,1}^{-1}a_{1,2}^{-1},  \quad a_{0,4}a_{1,2}a_{0,2}^{-1}a_{1,4}^{-1}, $$
$$a_{0,7}a_{1,7}a_{0,5}^{-1}a_{1,5}^{-1}, \quad a_{0,8}a_{1,5}a_{0,6}^{-1}a_{1,7}^{-1},  \quad a_{0,5}a_{1,8}a_{0,7}^{-1}a_{1,6}^{-1},  \quad a_{0,6}a_{1,6}a_{0,8}^{-1}a_{1,8}^{-1}. $$
\item The following $8$ squares get mapped to the square of $X_2$ attached along $a_2^{-1}a_0a_2a_1^{-1}$: $$a_{2,1}^{-1}a_{0,1}a_{2,7}a_{1,5}^{-1}, \quad a_{2,7}^{-1}a_{0,2}a_{2,1}a_{1,6}^{-1}, \quad a_{2,3}^{-1}a_{0,3}a_{2,5}a_{1,7}^{-1}, \quad a_{2,5}^{-1}a_{0,4}a_{2,3}a_{1,8}^{-1},$$
$$a_{2,2}^{-1}a_{0,7}a_{2,4}a_{1,1}^{-1}, \quad a_{2,4}^{-1}a_{0,8}a_{2,2}a_{1,2}^{-1}, \quad a_{2,8}^{-1}a_{0,5}a_{2,6}a_{1,3}^{-1}, \quad a_{2,6}^{-1}a_{0,6}a_{2,8}a_{1,4}^{-1}.$$\end{enumerate}
This completes the construction of $Q$, together with a combinatorial covering $Q \to X_2$.

Now let $Y$ be the first square subdivision of $Q$ and let $\mathcal{G}(Y)$ be the trivial complex of groups over $Y$.  We assign types $T \in \cS$ to the vertices of $Y$, as follows.  If a vertex of $Y$ is one of the vertices of $Q$, it has type~$\emptyset$.  Next consider the vertices of $Y$ which are midpoints of edges of $Q$. 
Table~\ref{table:edge types} shows the assigned types of these vertices.  To simplify notation, we write $i^{\pm}$ for the type $\{ s_{i^{\pm}} \} \in \cS$, for $i = 0,1,2$.  
\begin{table}
\caption{Types of vertices in $Y$ which are midpoints of edges of $Q$} \label{table:edge types}
\begin{tabular}{| c | c|| c | c|| c| c|}\hline
Midpoint of edges & Type & Midpoint of edges & Type & Midpoint of edges & Type\\ \hline
$a_{1,1}, a_{1,3}, a_{0,5}, a_{0,7}$ & $0^+$ & $a_{0,1}, a_{0,3}, a_{1,5}, a_{1,7}$ & $1^+$ & 
$a_{2,2}, a_{2,4}, a_{2,6}, a_{2,8}$ 
& $2^+$ \\ \hline
$a_{1,2}, a_{1,4}, a_{0,6}, a_{0,8}$ & $0^-$ & $a_{0,2}, a_{0,4}, a_{1,6}, a_{1,8}$ & $1^-$  & 
$a_{2,1}, a_{2,3}, a_{2,5}, a_{2,7}$ 
& $2^-$  \\ \hline
\end{tabular}
\end{table}
Finally  consider the vertices of $Y$ which are at the centres of squares of $Q$.  Let $\s$ be such a vertex.  Then for some pair of types $i^{\varepsilon_i}$ and $j^{\varepsilon_j}$ with $i, j \in \{0,1,2\}$, $i \neq j$, and $\varepsilon_i, \varepsilon_j \in \{\pm\}$, two of the vertices of $Y$ which are adjacent to $\s$ are of type $i^{\varepsilon_i}$, and  two of the vertices of $Y$ which are adjacent to $\s$ are of type $j^{\varepsilon_j}$.  Moreover, $\{ i^{\varepsilon_i}, j^{\varepsilon_j} \} \in \cS$.  We then assign type $\{ i^{\varepsilon_i}, j^{\varepsilon_j} \}$ to the vertex $\s$. 

After assigning these types, it may be verified that $Y$ is obtained by taking $8$ copies of the chamber $K_2$ and gluing together certain pairs of mirrors of the same type.  We note also that the above assignment of types allows us to orient the edges of $Y$ in the same way as in $K_2$, that is, an edge $a$ has initial vertex of type $T'$ and terminal vertex of type $T$ if and only if $T' \subsetneq T$.

Next, define $f:Y \to K_2$ to be the only possible type-preserving morphism.  It may be checked that $f$ is a nondegenerate morphism.  We  construct a covering of complexes of groups $\Phi:\mathcal{G}(Y) \to \mathcal{W}(K_2)$ over $f$.  
In order to define the elements $\phi(a)$ for the edges $a$ of $Y$, we put an equivalence relation, parallelism, on 
the set of edges of $Y$, so that if $a$ and $b$ are parallel then we will have $\phi(a) = \phi(b)$.  The relation is 
generated by saying that two edges are \emph{parallel} if they are opposite edges of a (small) square of $Y$.   
The values of $\phi(a)$ for representatives $a$ of certain of the parallelism classes of edges in $Y$ are specified 
in Table~\ref{table:phi(a)}.  For all edges $a$ of $Y$ which are \emph{not} parallel to an edge appearing in Table~
\ref{table:phi(a)}, we put $\phi(a) = 1$.

\begin{table}
\caption{Nontrivial values of $\phi(a)$, for representatives $a$ of certain parallelism classes of edges} \label{table:phi(a)}
\begin{tabular}{| c | c| c|| c | c| c|}\hline
Vertex $i(a)$ & Type of $t(a)$ & $\phi(a)$ & Vertex $i(a)$ & Type of $t(a)$ & $\phi(a)$ \\ \hline
$v_1$ & $0^+$ & $s_{0^+}$ & $v_1$ & $2^-$ & $s_{2^-}$ \\ \hline
$v_4$ & $0^-$ & $s_{0^-}$ & $v_4$ & $2^-$ & $s_{2^-}$ \\ \hline
$v_1$ & $1^+$ & $s_{1^+}$ & $v_5$ & $2^+$ & $s_{2^+}$ \\ \hline
$v_2$ & $1^-$ & $s_{1^-}$ & $v_6$ & $2^+$ & $s_{2^+}$\\ \hline
\end{tabular}
\end{table}

To verify that $\Phi$ is a covering of complexes of groups, we simplify notation and write $s$ for the vertex $\s_{\{s\}}$ of the chamber $K_2$.  For each vertex $s_{i^\varepsilon}$ of $K_2$, where $i \in \{0,1,2\}$ and $\varepsilon \in \{\pm\}$, there is a unique edge $b$ of $K_2$ such that $s_{i^\varepsilon}$ is the terminal vertex of $b$.  Fix a vertex $\s \in f^{-1}(s_{i^\varepsilon})$.  Then there are two edges $a_1$ and $a_2$ of $Y$ with terminal vertex $\sigma$ such that $f(a_1) = f(a_2) = b$.  By construction, without loss of generality we have $\phi(a_1) = s_{i^\epsilon}$ and $\phi(a_2) = 1$.  Therefore $\Phi_{\s/b}$ is a bijection to $\la s_{i^\varepsilon} \ra \cong C_2$ as required.  Now consider a vertex $\s_T$ of $K_2$ where $T \in \cS$ with $|T| = 2$.  Write $T = \{i^{\varepsilon_i}, j^{\varepsilon_j} \}$.  If $b$ is an edge of $K_2$ with terminal vertex $\sigma_T$, then without loss of generality $b$ has initial vertex of type $i^{\varepsilon_i}$.  Fix a vertex $\sigma \in f^{-1}(\s_T)$.  Then there are two edges $a_1$ and $a_2$ of $Y$ with terminal vertex $\sigma$ such that $f(a_1) = f(a_2) = b$.  By construction, without loss of generality we have $\phi(a_1) = s_{j^{\epsilon_j}}$ and $\phi(a_2) = 1$.  Thus $\Phi_{\s/b}$ is a bijection to $W_T/\la s_{i^{\varepsilon_i}} \ra \cong \la s_{j^{\varepsilon_j}} \ra \cong C_2$ as required.  Therefore $\Phi$ is a covering of complexes of groups.  \end{proof}

\subsection{Discussion of case $d > 2$}\label{sec:d>2}

We conclude with a discussion of whether $G_d$ and $W_d$ are commensurable when $d > 2$.

The first result in this section says that in order to use covering-theoretic arguments to find a common finite index subgroup of $G_d$ and $W_d$, it suffices to consider coverings $\Phi:\mathcal{G}(Y') \to \mathcal{W}(K_d)$ where $\mathcal{G}(Y')$ is a \emph{trivial} complex of groups.  We denote by $Z_d$ the first square subdivision of $X_d$.

\begin{lemma}\label{lem:no torsion}  Let $d \geq 2$.  Suppose that $Y$ and $Y'$ are finite square complexes, such that for some complex of groups $\mathcal{G}(Y')$ over $Y'$, all of the following hold:
\begin{enumerate}
\item there is a combinatorial covering map $\Psi:Y \to Z_d$; 
\item there is a covering of complexes of groups $\Phi:\mathcal{G}(Y') \to \mathcal{W}(K_d)$; and
\item the fundamental group of $Y$ is isomorphic to the fundamental group of $\mathcal{G}(Y')$.
\end{enumerate} 
Then $\mathcal{G}(Y')$ is the trivial complex of groups.
\end{lemma}

\begin{proof}  Since $G_d$ is torsion-free, these assumptions imply that the fundamental group of $\mathcal{G}(Y')$ is torsion-free.

Assume that $\mathcal{G}(Y')$ is not the trivial complex of groups.  Then there is a vertex $\s$ of $Y'$ such that the monomorphism $\phi_{\s}:G_\s \to W_T$ has nontrivial image, for some $T \in \cS$.  But $W_T$ is a finite group, and so $G_\s$ is finite.  Since $\mathcal{G}(Y')$ is developable, we thus have a nontrivial finite group $G_\s$ which embeds in the torsion-free group $\pi_1(\mathcal{G}(Y'))$.  This is a contradiction.
\end{proof}

We now show that the strategy used to prove that $G_2$ and $W_2$ are commensurable cannot be implemented for $d > 2$.  By Lemma~\ref{lem:no torsion}, we need only consider coverings $\Phi:\mathcal{G}(Y) \to \mathcal{W}(K_d)$ where $\mathcal{G}(Y)$ is the trivial complex of groups.  

\begin{prop}\label{prop:no covering}  If $d > 2$ there is no square complex $Y$ such that both of the following conditions hold:
\begin{enumerate}
\item there is a combinatorial covering map $\Psi:Y \to Z_d$; and
\item there is a covering of complexes of groups $\Phi:\mathcal{G}(Y) \to \mathcal{W}(K_d)$, where $\mathcal{G}(Y)$ is the trivial complex of groups over $Y$.
\end{enumerate}
\end{prop}

Note that this statement does not require $Y$ to be finite.  This shows that, in particular, the universal cover $\widetilde{X_d}$ is not isometric to $\Sigma_d$, even though both are $\CAT(0)$ square complexes with all vertex links the graph $\G_d$.  After proving Proposition \ref{prop:no covering}, we discuss the possibility that $G_d$ has finite index subgroup $\pi_1(Y)$ and $W_d$ has finite index subgroup $\pi_1(Y')$ with $Y \neq Y'$ but $\pi_1(Y)$ isomorphic to $\pi_1(Y')$.

\begin{proof}  Assume by contradiction that there is a square complex $Y$ such that both (1) and (2) hold.  

We will assign a type to each vertex $\s$ of $Y$ in two different ways, and then establish the relationship between these types in Corollary~\ref{cor:types} below.  First, if $\Psi(\sigma)$ is respectively a vertex, midpoint of an edge or centre of a face of $X_d$, then $\s$ has type respectively $V$, $E$ or $F$.   The following lemma is immediate from the construction of $X_d$.

\begin{lemma}\label{lem:types VEF}
Let $\s$ be a vertex of $Y$. 
\begin{enumerate}
\item If $\s$ is of type $V$ then $\s$ is contained in $4d$ small squares.
\item If $\s$ is of type $E$ then $\s$ is contained in: \begin{enumerate}\item $2(d+1)$ small squares if $\Psi(\s)$ is the midpoint of $a_0$;  \item $6$ small squares if $\Psi(\s)$ is the midpoint of $a_i$ for $1 \leq i < d$; and \item $4$ small squares if $\Psi(\s)$ the midpoint of $a_d$.\end{enumerate}
\item If $\s$ is of type $F$ then $\s$ is contained in $4$ small squares.
\end{enumerate}
\end{lemma}

Second, let the covering
$\Phi:\mathcal{G}(Y) \to \mathcal{W}(K_d)$ be over the nondegenerate morphism $f:Y
\to K_d$, and assign type $T \in \cS$ to each vertex $\s \in f^{-1}(\s_T)$.  Recall that $\Sigma$ may be viewed as the universal cover of the complex of groups $\mathcal{W}(K_d)$.   Since there is a covering $\Psi:\mathcal{G}(Y) \to \mathcal{W}(K_d)$, it follows that the complex of groups $\mathcal{G}(Y)$ is induced by the action of its fundamental group $\pi_1(\mathcal{G}(Y))$, which is a subgroup of $W_d$, on its universal cover, which is $\Sigma$.

\begin{lemma}\label{lem:nontrivial}
Let $\s$ be a vertex of $Y$, of type $T \in \cS$.  Then the number of small squares containing $\s$ in $Y$ is equal to the number of small squares containing each lift of $\s$ in $\Sigma$.
\end{lemma}

\begin{proof}  Since $\mathcal{G}(Y)$ is the trivial complex of groups, the quotient map $\Sigma \to Y$ induced by the action of $\pi_1(\mathcal{G}(Y))$ is a combinatorial covering map, and the result follows.
\end{proof}

Now let $\tau$ be a vertex of $\Sigma$ with the cellulation by small squares.  Then by construction:
\begin{itemize}
\item if $\tau$ has type $\emptyset$, then the link of $\tau$ is $\G_d$, and so $\tau$ is contained in $4d$ small squares;  \item if $\tau$ has type $T = \{ s \}$, and the panel containing $\tau$ is the star graph of valence $n$, equivalently the vertex $s$ of $\G_d$ has valence $n$, then $\tau$ is contained in $2n$ small squares; and \item if $\tau$ has type $T \in \cS$ with $|T| = 2$, then $\tau$ is contained in $4$ small squares.  
\end{itemize}

\begin{cor}\label{cor:types T}
Let $\s$ be a vertex of $Y$.
\begin{enumerate}
\item If $\s$ is of type $\emptyset$ then $\s$ is contained in $4d$ small squares.
\item If $\s$ is of type $T \in \cS$ with $|T| = 1$ then $\s$ is contained in: \begin{enumerate}\item $2(d+1)$ small squares if $f(\s) \in \{ s_{0^+}, s_{0^-}\}$;  \item $6$ small squares if $f(\s) \in \{ s_{i^+}, s_{i^-} \}$ for $1 \leq i < d$; and \item $4$ small squares if $f(\s) \in \{ s_{d^+}, s_{d^-}\}$.\end{enumerate}
\item If $\s$ is of type $T \in \cS$ with $|T| = 2$ then $\s$ is contained in $4$ small squares.
\end{enumerate}
\end{cor}

The relationship between the two type-systems is thus as follows.

\begin{cor}\label{cor:types} Let $\s$ be a vertex of $Y$.  Then:
\begin{enumerate}
\item  $\s$ has type $V$ if and only if it has type $\emptyset$;
\item  $\s$ has type $E$ if and only if it has type $T \in \cS$ with $|T| = 1$; and
\item  $\s$ has type $F$ if and only if it has type $T \in \cS$ with $|T| = 2$.
\end{enumerate}
\end{cor}

\begin{proof}  
Part (1) is immediate from Lemma~\ref{lem:types VEF} and Corollary~\ref{cor:types T}.  Parts (2) and (3) follow from these results, together with the observations that vertices of types $V$ and $F$ are never adjacent while every vertex of type $E$ is adjacent to at least one vertex of type $V$, and similarly for the types $T \in \cS$.
\end{proof}

To complete the proof of Proposition~\ref{prop:no covering}, fix a vertex $\sigma$ in $Y$ of type $V$, equivalently of type $\emptyset$, and consider the set of vertices of $Y$ which are adjacent to $\sigma$.   These vertices are all of type $E$, and so by part (2) of Lemma~\ref{lem:types VEF} they are all of type $T \in \cS$ with $|T| = 1$.  Since $f$ is a nondegenerate morphism and $G_\s$ is trivial, the restriction of $f$ to the set of edges of $Y$ which are incident to $\s$ is a bijection to the set of edges of $K_d$ which are incident to $\s_\emptyset$, and this bijection of edges induces a bijection from the set of vertices adjacent to $\s$ to the set of vertices adjacent to $\s_\emptyset$.  In particular, for each vertex adjacent to $\s$, there is a unique edge of $Y$ containing both $\s$ and this vertex.  By definition of types $T \in \cS$, it also follows that no two vertices adjacent to $\s$ have the same type $T \in \cS$.

Therefore, by part (2)(c) of Corollary~\ref{cor:types T}, there are exactly $2$ vertices adjacent to $\s$ which are contained in exactly $4$ small squares of $Y$, and we may denote these vertices by $\s_{d^+}$ and $\s_{d^-}$ where without loss of generality $f(\s_{d^+}) = s_{d^+}$ and $f(\s_{d^-}) = s_{d^-}$.  Similarly, since $d > 2$, there are exactly $2$ vertices adjacent to $\s$ which are contained in exactly $2(d+1)$ small squares of $Y$, and we may denote them by $\s_{0^+}$ and $\s_{0^-}$, such that $f(\s_{0^+}) = s_{0^+}$ and $f(\s_{0^-}) = s_{0^-}$.   

By part (2) of Corollary~\ref{cor:types} and part (2)(c) of Lemma~\ref{lem:types VEF}, $\Psi(\s_{d^+})$ is the midpoint of the edge $a_d$ of $X_d$.   Now there is a unique edge of $Y$ containing both $\s$ and $\s_{d_+}$, and so as $\Psi$ is a combinatorial covering map, there is a unique vertex $\s'$ in $Y$ which is of type $V$, is adjacent to $\s_{d^+}$ and is not equal to $\s$.   There are thus $2$ small squares in $Y$ which contain both $\s_{d^+}$ and $\s'$. 

Let $\s_E$ be the unique vertex in $Y$ of type $E$ such that $\s_E$ and $\s_{d^+}$ are both adjacent to $\s'$, and $\s_E$, $\s_{0^+}$ and $\s_{d^+}$ are all adjacent to the same vertex, say $\s_F$, of type $F$.  Now, by Lemma~\ref{lem:types VEF}, we have that $\Psi(\s_{0^+})$ is the midpoint of $a_0$ and $\Psi(\s_{d^+})$ is the midpoint of $a_d$.  Thus $\Psi(\s_F)$ is the centre of the only big square in $X_d$ with both $a_0$ and $a_d$ as edges, namely the square with boundary label $a_d^{-1}a_0 a_{d} a_{d-1}^{-1}$.  Therefore $\Psi(\s_E)$ is the midpoint of $a_{d-1}$.  Thus the vertex $\s_E$ is contained in $6$ small squares, and so by part (2)(b) of Corollary~\ref{cor:types T}, since $d > 2$ we have $f(\s_E) \in \{ s_{i^+}, s_{i^-} \}$ for some $1 \leq i < d$.  That is, $\s_E$ is of type $i^+$ or $i^-$ for some $1 \leq i < d$.

By part (3) of Corollary~\ref{cor:types}, the vertex $\s_F$ has type some $T \in \cS$ with $|T| = 2$.  Since $\s_F$ is adjacent to vertices of types $0^+$ and $d^+$, the vertex $\s_F$ is of type $\{0^+,d^+\}$.  But $\s_F$ is also adjacent to vertices of types $d^+$ and $i^\pm$ with $1 \leq i < d$, so $\s_F$ is of type $\{i^\pm, d^+\}$.  This is impossible.  Therefore there is no square complex $Y$ such that both (1) and (2) of Proposition~\ref{prop:no covering} hold.
\end{proof}

For $d > 2$ we do not know if $G_d$ and $W_d$ are commensurable, or even quasi-isometric.  If they are commensurable, then there are finite square complexes $Y$ and $Y'$ with isomorphic fundamental groups, such that there is a combinatorial covering map $Y \to Z_d$ and a covering of complexes of groups $\mathcal{G}(Y') \to \mathcal{W}(K_d)$, with $\mathcal{G}(Y')$ the trivial complex of groups over $Y'$ by Lemma~\ref{lem:no torsion}.  By Proposition~\ref{prop:no covering}, we know that $Y \neq Y'$ and that the universal covers of $Y$ and $Y'$ are not isometric.  Hence if there is some Mostow-type rigidity result which implies that the isomorphism $\pi_1(Y) \cong \pi_1(Y')$ is induced by an isometry of universal covers, we would obtain that $W_d$ and $G_d$ are not in fact commensurable.  However, the only Mostow-type  rigidity results  for $\CAT(0)$ square complexes that we know of are Theorem 1.4.1 of~\cite{burger-mozes-zimmer}, for certain uniform lattices on products of trees, and Corollary 1.8 of~\cite{bestvina-kleiner-sageev-RAAG}, concerning right-angled Artin groups, and neither of these results can be applied here.

\bibliographystyle{siam}
\bibliography{refs}

\end{document}